\documentclass[12pt]{article}

\usepackage{amssymb}
\usepackage{amsmath}
\usepackage{amscd}
\usepackage[all]{xy}
\usepackage{amsthm}
\usepackage{graphicx}
\usepackage{mathrsfs}
\usepackage[margin=30truemm]{geometry}
\usepackage{enumitem}

\newtheorem{dfn}{Definition}[section]
\newtheorem{thm}[dfn]{Theorem}
\newtheorem{prop}[dfn]{Proposition}
\newtheorem{lem}[dfn]{Lemma}
\newtheorem{cor}[dfn]{Corollary}

\newtheorem*{cor*}{Corollary}

\theoremstyle{definition}
\newtheorem{rem}[dfn]{Remark}
\newtheorem{exa}[dfn]{Example}

\newcommand{\K}{K\"ahler }

\newcommand{\X}{\mathscr{X}}
\newcommand{\XX}{\mathscr{X}^{\iota}}
\newcommand{\xs}{\mathscr{X}/S}
\newcommand{\xss}{\mathscr{X}^{\iota}/S}

\newcommand{\elm}{\operatorname{elm}}
\newcommand{\mon}{\operatorname{Mon}^2}
\newcommand{\ktm}{\operatorname{KT}(M)}
\newcommand{\ktmo}{\operatorname{KT}(M_0)}
\newcommand{\vol}{\operatorname{Vol}}

\newcommand{\rk}{\operatorname{rk}}
\newcommand{\sign}{\operatorname{sign}}

\newcommand{\G}{\mu_2}

\newcommand{\Address}{{
  \bigskip
  \footnotesize

  \textsc{Department of Mathematics, Faculty of Science, Kyoto University, Kyoto 606-8502, Japan}\par\nopagebreak
  \textit{E-mail address}: \texttt{imaike.dai.22s@st.kyoto-u.ac.jp}

}}

\newcommand{\ax}{c_1(\overline{T}{\xss})}
\newcommand{\ay}{c_2(\overline{T}{\xss})}
\newcommand{\az}{c_1(\overline{N}_{\mathscr{X}^{\iota}/ \mathscr{X}})}
\newcommand{\aw}{c_2(\overline{N}_{\mathscr{X}^{\iota}/ \mathscr{X}})}
\newcommand{\au}{c_1(\overline{T}{\xs})|_{\mathscr{X}^{\iota}}}
\newcommand{\av}{c_2(\overline{T}{\xs})|_{\mathscr{X}^{\iota}}}

\begin{document} 

\title{Analytic torsion for irreducible holomorphic symplectic fourfolds with involution, I: Construction of an invariant}
\author{Dai Imaike}
\date{}

\maketitle

\numberwithin{equation}{section}

\begin{abstract}
	In this paper, we construct an invariant for irreducible holomorphic symplectic manifolds of $K3^{[2]}$-type with antisymplectic involution by using the equivariant analytic torsion.
	Moreover, we give a formula for the complex Hessian of the logarithm of the invariant.
\end{abstract}

\tableofcontents

\setcounter{section}{-1}

\section{Introduction}\label{s-0}

	The notion of analytic torsion for complex manifolds was introduced by Ray-Singer \cite{MR383463}.
	It is the exponential of the derivative at zero of the weighted alternating sum of various spectral zeta functions.
	In \cite{MR783704}, Quillen defined a metric on the determinant of cohomologies by using analytic torsion.
	He introduced the product of the $L^2$-metric on the determinant of cohomologies and the analytic torsion,
	and this metric is called the Quillen metric.
	One of the remarkable properties of the Quillen metric is that it is a $C^{\infty}$ metric on the determinant of cohomologies even if the dimension of the cohomologies jumps. 
	Bismut-Gillet-Soul\'e \cite{MR929146} \cite{MR929147} \cite{MR931666} calculated the curvature of the determinant of cohomologies endowed with the Quillen metric 
	for locally K\"ahler proper holomorphic submersions.
	In particular, they established the Grothendieck-Riemann-Roch theorem at the level of differential forms at bidegree $(1, 1)$.
	Gillet-Soul\'e \cite{MR1189489} proved the arithmetic Riemann-Roch theorem in Arakelov geometry, where the determinant of cohomologies are endowed with the Quillen metric.
	
	In theoretical physics, Bershadsky-Cecotti-Ooguri-Vafa \cite{MR1301851} introduced a weighted alternating product of various analytic torsions for Calabi-Yau manifolds.
	This special combination of analytic torsions is called the BCOV torsion.
	One of their predictions is an equivalence of the BCOV torsion in B-model and the genus one Gromov-Witten invariants in the corresponding A-model.
	In mathematics, Fang-Lu-Yoshikawa \cite{MR2454893} constructed the BCOV invariant of Calabi-Yau threefolds, which could be viewed as a normalization of the BCOV torsion.
	Eriksson-Freixas i Montplet-Mourougane \cite{MR4255041}, \cite{MR4475251} constructed the BCOV invariant of Calabi-Yau manifolds of arbitrary dimension
	and established mirror symmetries at genus one for the pair of Calabi-Yau manifolds consisting of the projective hypersurface and its mirror family.
	Fu-Zhang \cite{MR4549963} constructed the BCOV invariant of Calabi-Yau varieties with canonical singularities and proved its birational invariance.
	
	In \cite{MR1316553}, Bismut studied analytic torsion in the equivariant setting and obtained the embedding formula and anomaly for equivariant Quillen metrics.
	In \cite{MR1800127}, Ma proved the curvature formula for equivariant Quillen metrics.
	K\"ohler and R\"ossler \cite{MR1872550} proved the fixed point formula of Lefschetz type in Arakelov geometry, which is an equivariant analog of the arithmetic Riemann-Roch theorem in \cite{MR1189489}.
	
	Yoshikawa \cite{MR2047658} constructed an invariant of 2-elementary K3 surfaces by using equivariant analytic torsion.
	He proved that for each deformation type the invariant is expressed as the Petersson norm of a certain automorphic form on a bounded symmetric domain of type IV and a certain Siegel modular form.
	Furthermore, it gives the BCOV invariants of Borcea-Voisin manifolds and log-Enriques surfaces \cite{MR3821932}, \cite{MR4476107}.
	
	The goal of my research is to generalize this result of Yoshikawa \cite{MR2047658} to a class of higher dimensional manifolds. 
	In this paper, we construct an invariant of irreducible holomorphic symplectic manifolds of $K3^{[2]}$-type with antisymplectic involution by using the equivariant analytic torsion of the holomorphic cotangent bundle. 
	Furthermore, we provide a variational formula for the invariant.
	Let us explain our results in more details.

	\bigskip
	
	A simply-connected compact \K manifold $X$ is an irreducible holomorphic symplectic manifold if there exists an everywhere non-degenerate holomorphic 2-form $\eta$ on $X$ such that $H^0(X, \Omega^2_X)$ is generated by $\eta$.
	The dimension of an irreducible holomorphic symplectic manifold is even.
	A 2-dimensional irreducible holomorphic symplectic manifold is a K3 surface.
	Beauville (\cite{MR730926}) proved that the Hilbert scheme of length $n$ zero-dimensional subschemes of a K3 surface is a $2n$-dimensional irreducible holomorphic symplectic manifold.  
	An irreducible holomorphic symplectic manifold is of $K3^{[n]}$-type if it is deformation equivalent to the Hilbert scheme of $n$-points of a K3 surface.
	An involution $\iota : X \to X$ is antisymplectic if $\iota^* \eta = -\eta$.
	
	Let $X$ be a manifold of $K3^{[2]}$-type and let $\iota : X \to X$ be an antisymplectic involution.
	The cohomology $H^2(X, \mathbb{Z})$ is equipped with an integral symmetric non-degenerate quadric form $q_X$, called the Beauville-Bogomolov-Fujiki form.
	Then $(H^2(X, \mathbb{Z}), q_X)$ is isomorphic to $L_2 = L_{K3} \oplus \mathbb{Z}e$ as lattices, where $L_{K3}$ is the K3 lattice and $\mathbb{Z}e$ is a rank $1$ lattice generated by $e$ with $e^2=-2$.
	An isometry $\alpha : H^2(X, \mathbb{Z}) \to L_2$ is called a marking.
	
	Let $\mon(L_2)$ be the subgroup of the isometry group $O(L_2)$ defined by $\mon(L_2) = \alpha \circ \mon(X) \circ \alpha^{-1}$,
	where $(X, \alpha)$ is a marked manifold of $K3^{[2]}$-type and $\mon(X)$ is the monodromy group of $X$ (\cite{MR2964480}).
	The group $\mon(L_2)$ is independent of the choice of $(X,\alpha)$.
	An admissible sublattice of $L_2$ is the pair of a hyperbolic sublattice $M$ of $L_2$ and an involution $\iota_M \in \mon(L_2)$ such that the invariant subspace of $\iota_M$ coincides with $M$.
	Set
	$$
		\tilde{\mathcal{C}}_M =\{ x \in M_{\mathbb{R}} ; x^2>0 \} \text{ and } \Delta(M)= \{ \delta \in M; \delta^2=-2, \text{ or } \delta^2=-10, (\delta, L_2)=2\mathbb{Z} \}.
	$$ 
	A K\"ahler-type chamber is a connected component of the set $\tilde{\mathcal{C}}_M \setminus \cup_{\delta \in \Delta(M)} \delta^{\perp}$.
	Joumaah (\cite{MR3519981}) showed that the deformation types of manifolds of $K3^{[2]}$-type with antisymplectic involution are classified by the admissible sublattices and the K\"ahler-type chambers. 
	
	Let $(M, \iota_M)$ be an admissible sublattice and let $\mathcal{K}$ be a K\"ahler-type chamber.
	Let $(X, \iota)$ be a manifold of $K3^{[2]}$-type with antisymplectic involution.
	We assume that there is an isometry $\alpha : H^2(X, \mathbb{Z}) \to L_2$ such that $\iota_M \circ \alpha = \alpha \circ \iota^*$ and $\alpha(\mathcal{K}^{\iota}_X) \subset \mathcal{K}$, where $\mathcal{K}^{\iota}_X$ is an invariant K\"ahler cone of $(X, \iota)$.
	We call such a pair $(X, \iota)$ a manifold of $K3^{[2]}$-type with antisymplectic involution of type $(M, \mathcal{K})$.
	
	Let $\omega_X$ be an $\iota$-invariant \K form on $X$.
	The volume of $(X, \omega_X)$ is defined by $\vol(X, \omega_X)= \int_X \frac{\omega_X^4}{4!}$.
	We denote the fixed locus of $\iota : X \to X$ by $X^{\iota}$.
	This is a possibly disconnected smooth complex surface.
	If $X^{\iota}=\sqcup_{i} Z_i$ is the decomposition of $X^{\iota}$ into the connected components, 
	then the volume of $(X^{\iota}, \omega_X|_{X^{\iota}})$ is defined by 
	$\vol(X^{\iota}, \omega_X|_{X^{\iota}}) =\prod_i \int_{Z_i} \frac{\omega_X|_{Z_i}^2}{2!}$.
	We denote the covolume of the lattice $\operatorname{Im} \left( H^1(X^{\iota}, \mathbb{Z}) \to H^1(X^{\iota}, \mathbb{R}) \right)$ in the cohomology $H^1(X^{\iota}, \mathbb{R})$ by $\vol_{L^2}\left( H^1(X^{\iota}, \mathbb{Z}), \omega_X|_{X^{\iota}} \right)$.
	
	We define the real-valued smooth function $\varphi$ on $X^{\iota}$ by
	$$
		\varphi=\frac{|| \eta^2 ||_{L^2}^2}{\eta^2 \wedge \bar{\eta}^2}\frac{\omega_X^4/4!}{\vol(X, \omega_X)}.
	$$
	This is independent of the choice of $\eta$.
	The positive real number $A(X, \iota, \omega_X)$ is defined by
	$$
		A(X, \iota, \omega_X) = \exp \left[ \int_{X^{\iota}} (\log \varphi) \Omega \right],
	$$
	where $\Omega$ is the characteristic form on $X^{\iota}$ defined by
	$$
		\Omega = c_1(TX^{\iota}, h_{X^{\iota}})^2 -8c_2(TX^{\iota}, h_{X^{\iota}}) -c_1(TX, h_X)|^2_{X^{\iota}} +3c_2(TX, h_X)|_{X^{\iota}}.
	$$
	
	The $\iota$-invariant \K form $\omega_X$ on $X$ induces the hermitian metric $h_X$ on the holomorphic cotangent bundle $\Omega^1_X$,
	and we set $\bar{\Omega}^1_X = (\Omega^1_X, h_X)$.
	The equivariant analytic torsion of $\bar{\Omega}^1_X$ is denoted by $\tau_{\iota}(\bar{\Omega}^1_X)$.
	We denote by $\bar{\mathcal{O}}_{X^{\iota}}$ the trivial line bundle $\mathcal{O}_{X^{\iota}}$ equipped with the canonical metric.
	The analytic torsion of $\bar{\mathcal{O}}_{X^{\iota}}$ is denoted by $\tau(\bar{\mathcal{O}}_{X^{\iota}})$.
	
	Set $t = \operatorname{Tr}(\iota_M)+2$.
	We define a real number $\tau_{M, \mathcal{K}}(X, \iota)$ by
	\begin{align*}
		\tau_{M, \mathcal{K}}(X, \iota)=\tau_{\iota}(\bar{\Omega}_X^1) &\vol(X, \omega_{X})^{\frac{(t-1)(t-7)}{16}} A(X, \iota, h_X) \\
		&\cdot \tau(\bar{\mathcal{O}}_{X^{\iota}})^{-2} \vol(X^{\iota}, \omega_{X^{\iota}})^{-2} \vol_{L^2}(H^1(X^{\iota}, \mathbb{Z}), \omega_{X^{\iota}}).
	\end{align*}
	As an application of the curvature formula for Quillen metrics \cite{MR929146}, \cite{MR1800127}, we have the following.
	
	\begin{thm}\label{t-0-1}
		The real number $\tau_{M, \mathcal{K}}(X, \iota)$ is independent of the choice of an $\iota$-invariant \K form.
		In particular, $\tau_{M, \mathcal{K}}(X, \iota)$ is an invariant of $(X, \iota)$.
	\end{thm}
	
	Let $\tilde{\mathcal{M}}_{M, \mathcal{K}}$ be the set of isomorphism classes of $K3^{[2]}$-type manifolds with antisymplectic involution of type $(M, \mathcal{K})$.
	Joumaah (\cite{MR3519981}) constructed an orthogonal modular variety $\mathcal{M}_{M, \mathcal{K}}$, a reduced divisor $\bar{\mathscr{D}}_{M^{\perp}}$ on $\mathcal{M}_{M, \mathcal{K}}$, and the period map $P_{M, \mathcal{K}} : \tilde{\mathcal{M}}_{M, \mathcal{K}} \to \mathcal{M}_{M, \mathcal{K}}$
	such that $P_{M, \mathcal{K}} \left( \tilde{\mathcal{M}}_{M, \mathcal{K}} \right)= \mathcal{M}_{M, \mathcal{K}} \setminus \bar{\mathscr{D}}_{M^{\perp}}$. 
	We set $\mathcal{M}^{\circ}_{M, \mathcal{K}} = \mathcal{M}_{M, \mathcal{K}} \setminus \bar{\mathscr{D}}_{M^{\perp}}$.
	By Theorem \ref{t-0-1} and Joumaah's theorem \cite{MR3519981}, $\tau_{M, \mathcal{K}}$ is viewed as a smooth real-valued function on $\mathcal{M}^{\circ}_{M, \mathcal{K}}$.
	Namely,
	$$
		\tau_{M,\mathcal{K}}(p) = \tau_{M,\mathcal{K}}(X, \iota) \quad ((X, \iota) \in P_{M, \mathcal{K}}^{-1}(p) )
	$$
	is independent of the choice of $(X, \iota) \in P_{M, \mathcal{K}}^{-1}(p)$.
	
	Let $\omega_{\mathcal{M}_{M, \mathcal{K}}}$ be the orbifold \K form on $\mathcal{M}_{M, \mathcal{K}}$ induced from the \K form of the Bergman metric on the period domain
	$$
		\Omega_{M^{\perp}} = \{ [\eta] \in \mathbb{P}(M^{\perp}_{\mathbb{C}}) ; (\eta, \eta)=0, (\eta, \bar{\eta}) >0 \}.
	$$
	
	In Lemma \ref{l6-3-1}, we will prove the existence of a smooth $(1,1)$-form $\sigma_{M, \mathcal{K}}$ on $\mathcal{M}^{\circ}_{M, \mathcal{K}}$ such that for any $(X, \iota) \in \tilde{\mathcal{M}}_{M, \mathcal{K}}$ we have
	$$
		P_{M, \mathcal{K}}^*\sigma_{M, \mathcal{K}} = c_1(\pi_*\Omega^1_{\XX / \operatorname{Def}(X, \iota)}, h_{L^2}) -c_1(R^1\pi_*\mathcal{O}_{\mathscr{X}^{\iota}}, h_{L^2}) -2c_1(\pi_*K_{\XX / \operatorname{Def}(X, \iota)}, h_{L^2}),
	$$ 
	where $P_{M, \mathcal{K}} : \operatorname{Def}(X, \iota) \to \mathcal{M}_{M, \mathcal{K}}$ is the period map of the Kuranishi family $\pi : (\X, \iota) \to \operatorname{Def}(X, \iota)$ of $(X, \iota)$.
	
	\begin{thm}\label{t-0-A}
		The following equation of differential forms on $\mathcal{M}^{\circ}_{M, \mathcal{K}}$ holds:
		$$
			-dd^c \log \tau_{M, \mathcal{K}} = \frac{(t+1)(t+7)}{8} \omega_{\mathcal{M}_{M, \mathcal{K}}} +\sigma_{ M, \mathcal{K} }.
		$$
	\end{thm}
	
	There is an application of this invariant to families of $K3^{[2]}$-type manifolds with involution.
	
	\begin{thm}\label{t4-0-1}
		Suppose that $t \neq -7, -1$ and $q(X^{\iota})=p_g(X^{\iota})=0$ for each $(X, \iota) \in \tilde{\mathcal{M}}_{M, \mathcal{K}}$.
		There exists no irreducible projective curve on $\mathcal{M}^{\circ}_{M, \mathcal{K}}$.
		In particular, if $f : (\mathscr{X}, \iota) \to S$ is a family of $K3^{[2]}$-type manifolds with antisymplectic involution of type $(M, \mathcal{K})$,
		and if $S$ is compact, then $f$ is isotrivial.
		Namely any two fibers of $f$ are isomorphic.
	\end{thm}
	
	In \cite{MR2047658}, Yoshikawa uses equivariant analytic torsion of trivial line bundle on a 2-elementary K3 surface.
	If we consider equivariant analytic torsion of trivial line bundle on a manifold of $K3^{[2]}$-type with antisymplectic involution, we can construct an invariant in the same way.
	The variation formula for this another invariant is trivial, unlike Theorem \ref{t-0-A}.
	In some special cases, it can be proved that this another invariant is constant and we cannot construct a nontrivial invariant by using the equivariant analytic torsion of trivial line bundle.
	For this reason, we consider equivariant analytic torsion of cotangent bundle instead of trivial line bundle.
	For more detail, see Remark \ref{r5-3-3}.
	
	This is the first of a series of three papers investigating equivariant analytic torsion for manifolds of $K3^{[2]}$-type with antisymplectic involution. 
	The second paper \cite{I2} analyzes the singular behavior of the invariant and shows that, in some special cases, it is expressed as the Petersson norm of a certain automorphic form on a bounded symmetric domain of type IV and a certain Siegel modular form. 
	The third paper \cite{I3} compares the invariant constructed in this first paper with the BCOV invariant of the Calabi-Yau fourfold obtained as a crepant resolution of the quotient of a $K3^{[2]}$-type manifold by the antisymplectic involution.

	\bigskip
	
	Acknowledgements.
	I am grateful to my supervisor Ken-Ichi Yoshikawa, who suggests this topic to me and provides me with a lot of ideas.
	This work was supported by JST, the establishment of university the creation of science technology innovation, Grant Number JPMJFS2123
	and by JSPS KAKENHI Grant Number 23KJ1249.

\section{Analytic torsion and its fundamental properties}\label{s-1}

\subsection{Equivariant analytic torsion}\label{ss-1-1}

	In this section, we recall equivariant analytic torsion for compact \K manifolds with holomorphic involution.
	
	Let $X$ be a compact complex manifold of dimension $n$,
	and let $\iota : X \to X$ be a holomorphic involution of $X$.
	Let $\G$ be the group generated by the order 2 element $\iota$.
	In what follows, we consider the $\mu_2$-action on $X$ induced by $\iota$.
	Let $h_X$ be an $\iota$-invariant \K metric on $X$.
	The \K form attached to $h_X$ is defined by
	$$
		\omega_X = \frac{i}{2} \sum_{j,k} h_X \left( \frac{\partial}{\partial z^j}, \frac{\partial}{\partial z^k} \right) dz^j \wedge d\bar{z}^k ,
	$$
	where $z^1, \dots, z^n$ is a system of local coordinates on $X$.
	The space of smooth $(p,q)$-forms on $X$ is denoted by $A^{p,q}(X)$.
	
	Let $E$ be a $\G$-equivariant holomorphic vector bundle on $X$,
	and $h_E$ a $\G$-invariant hermitian metric on $E$.
	The space of $E$-valued smooth $(p,q)$-forms on $X$ is denoted by $A^{p,q}(X, E)$ or $A^{p,q}(E)$.
	
	The metrics $h_X$ and $h_E$ induce a $\G$-invariant hermitian metric $h$ on the complex vector bundle $\wedge^{p,q}T^*X \otimes E$.
	The $L^2$-metric on $A^{p,q}(X, E)$ is defined by
	$$
		\langle \alpha, \beta \rangle_{L^2} = \int_X h(\alpha, \beta) \frac{\omega_X^n}{n !}, \quad \alpha, \beta \in A^{p,q}(X, E).
	$$
	The Dolbeault operator of $E$ is denoted by $\bar{\partial}_E : A^{p,q}(X, E) \to A^{p,q+1}(X, E)$,
	and its formal adjoint is denoted by $\bar{\partial}_E^* : A^{p,q}(X, E) \to A^{p,q-1}(X, E)$.
	We define the Laplacian $D_{p,q}^2$ acting on $A^{p,q}(X, E)$ by 
	$$
		D_{p,q}^2 = (\bar{\partial}_E + \bar{\partial}_E^*)^2 : A^{p,q}(X, E) \to A^{p,q}(X, E).
	$$
	We denote the spectrum of $D_{p,q}^2$ by $\sigma(D_{p,q}^2)$, and the eigenspace of $D_{p,q}^2$ associated with an eigenvalue $\lambda \in \sigma(D_{p,q}^2)$ by $E_{p,q}(\lambda)$.
	Note that $\sigma(D_{p,q}^2)$ is a discrete subset contained in $\mathbb{R}_{\geqq 0}$.
	Moreover $E_{p,q}(\lambda)$ is finite dimensional.
	
	\begin{dfn}\label{d-1-1}
		Let $g \in \G$. 
		The spectral zeta function is defined by 
		$$
			\zeta_{p,q,g}(s) = \sum_{\lambda \in \sigma(D_{p,q}^2) \setminus \{ 0 \}} \lambda^{-s} \operatorname{Tr} (g|_{E_{p,q}(\lambda)}) \quad (s \in \mathbb{C}, \operatorname{Re} s >n).
		$$
	\end{dfn}
	
	Note that $\zeta_{p,q,g}(s)$ converges absolutely on the domain $\operatorname{Re} s > n$
	and extends to a meromorphic function on $\mathbb{C}$ which is holomorphic at $s=0$.
	
	\begin{dfn}\label{d-1-2}
		Let $g \in \G$.
		The equivariant analytic torsion of $\overline{E} = (E, h_E)$ on $(X, \omega_X)$ is defined by
		$$
			\tau_g(\overline{E}) = \exp \left\{ - \sum_{q=0}^n (-1)^q q \zeta'_{0,q,g}(0) \right\} .
		$$
	\end{dfn}
	
	If $g=1$, it is the (usual) analytic torsion of $\overline{E}$ and is denoted by $\tau(\overline{E})$ instead of $\tau_1(\overline{E})$.
	
	We denote by $H^q(X, E)_{\pm}$ the $(\pm 1)$-eigenspace of $\iota^* : H^q(X, E) \to H^q(X, E)$.
	We set
	$$
		\lambda_{\pm}(E) = \bigotimes_{q \geqq 0} ( \det H^q(X, E)_{\pm} )^{(-1)^q}. 
	$$ 
	We define the equivariant determinant of the cohomologies of $E$ by
	$$
		\lambda_{\G}(E) = \lambda_{+}(E) \oplus \lambda_{-}(E). 
	$$
	
	By Hodge theory, we may identify $H^q(X, E)$ with the space of $E$-valued harmonic $(0, q)$-forms on $X$.
	The cohomology $H^q(X, E)$ is endowed with the $\G$-invariant hermitian metric induced from the $L^2$-metric on $A^{p,q}(X, E)$.
	It induces the hermitian metric $|| \cdot ||_{\lambda_{\pm}(E), L^2}$ on $\lambda_{\pm}(E)$.
	We define the equivariant metric on $\lambda_{\G}(E)$ by
	$$
		|| \alpha ||_{\lambda_{\G}(E), L^2}(\iota) = || \alpha_+ ||_{\lambda_{\pm}(E), L^2} \cdot || \alpha_- ||_{\lambda_{\pm}(E), L^2}^{-1} \quad ( \alpha=(\alpha_+, \alpha_-) \in \lambda_{\G}(E) ),
	$$
	and call it the equivariant $L^2$-metric on $\lambda_{\G}(E)$.
	We define the equivariant Quillen metric on $\lambda_{\G}(E)$ by
	$$
		|| \alpha ||^2_{\lambda_{\G}(E), Q}(\iota) = \tau_g(\overline{E}) || \alpha ||^2_{\lambda_{\G}(E), L^2}(\iota) .
	$$

\subsection{A fundamental property of equivariant analytic torsion}\label{ss-1-2}

	Let $\mathscr{X}$ and $S$ be complex manifolds of dimension $m+n$ and $m$, respectively.
	Let $\iota : \mathscr{X} \to \mathscr{X}$ be a holomorphic involution.
	Then $\iota$ induces a $\G$-action on $\mathscr{X}$.
	We consider the trivial $\G$-action on $S$.
	
	Let $f : (\mathscr{X}, \iota) \to S$ be a proper surjective $\G$-equivariant holomorphic submersion.
	Suppose that $f$ is locally K\"ahler.
	Namely, for each point $s \in S$ there is an open neighborhood $U$ of $s$ such that $f^{-1}(U)$ is K\"ahler.
	The fiber of $f$ is denoted by $X_s$ $(s \in S)$ or simply $X$.
	Since $f$ is $\G$-equivariant, the involution $\iota$ induces a holomorphic involution on each fiber $X_s$,
	which is denoted by $\iota_s$ or simply $\iota$. 
	
	Let $h_{\xs}$ be an $\iota$-invariant hermitian metric on the relative tangent bundle $T\xs$ which is fiberwise K\"ahler.
	Set $h_s = h_{\xs}|_{X_s}$ $(s \in S)$.
	This is an $\iota_s$-invariant \K metric on $X_s$.
	Its \K form is denoted by $\omega_s$
	and we set $\omega_{\xs}= \{ \omega_s \}_{s \in S}$.
	
	Let $\overline{E}=(E, h_E)$ be a $\G$-equivariant holomorphic hermitian vector bundle on $\X$.
	We assume that $R^qf_*E$ is a locally free sheaf for all $q \geqq 0$ and we may regard it as a holomorphic vector bundle on $S$.
	By Hodge theory, $R^qf_*E$ is equipped with the $\iota$-invariant hermitian metric.
	This is called the $L^2$-metric and denoted by $h_{L^2}$.
	
	Let $g \in \G$.
	We define a real-valued function on $S$ by
	$$
		\tau_g(\overline{E})(s) = \tau_g(\overline{E}|_{X_s} ) \quad (s \in S).
	$$
	
	Let $E_{\pm}$ be the $(\pm 1)$-eigenbundle of the $\G$-action on $E|_{\mathscr{X}^{\iota}}$,
	and the restriction of $h_E$ to $E_{\pm}$ is denoted by $h_{\pm}$.
	The curvature form of $(E_{\pm}, h_{\pm})$ is denoted by $R_{\pm}$.
	Recall that the equivariant Todd form and the equivariant Chern character form are differential forms on $\mathscr{X}^{\iota}$ defined by
	\begin{align}\label{al-1-A}
		Td_{\iota}(E, h_E) = Td \left(- \frac{R_+}{2\pi i} \right) \det \left( \frac{I}{I+ \exp({+\frac{R_-}{2\pi i}})} \right),
	\end{align}
	and
	\begin{align}\label{al-1-B}
		ch_{\iota}(E, h_E) = ch \left(- \frac{R_+}{2\pi i} \right) - ch \left(- \frac{R_-}{2\pi i} \right) ,
	\end{align}
	respectively.
	If $\alpha$ is a differential form, then $[\alpha]^{(p,q)}$ is the component of $\alpha$ of bidegree $(p,q)$.
	
	\begin{thm}\label{p-1-3}
		For each $g \in \G$, $\tau_g(\overline{E})$ is smooth on $S$.
		Moreover, it satisfies the following equation:
		$$
			-dd^c \log \tau_g(\overline{E}) +\sum_{q \geqq 0} (-1)^q [ch_g(R^qf_*E, h_{L^2})]^{(1,1)}
			=[f_* Td_g(T\xs, h_{\xs}) ch_g(\overline{E})]^{(1,1)}.
		$$
	\end{thm}
	
	\begin{proof}
		See \cite[Theorem 0.1]{MR929146} and \cite[Theorem 2.12.]{MR1800127}.
	\end{proof}
	
	The $(\pm 1)$-eigenbundle of $R^qf_*E$ is denoted by $(R^qf_*E)_{\pm}$.
	We set
	$$
		\lambda_{\pm}(E) = \bigotimes_{q \geqq 0} \det (R^qf_*E)_{\pm}.
	$$
	We define the equivariant determinant of the cohomologies of $E$ by
	$$
		\lambda_{\G}(E) = \lambda_{+}(E) \oplus \lambda_{-}(E).
	$$
	It is equipped with the equivariant $L^2$-metric and the equivariant Quillen metric.
	For an open subset $U$ of $S$, a holomorphic section $\sigma = (\sigma_+, \sigma_-)$ is called an admissible section if both $\sigma_+$ and $\sigma_-$ are nowhere vanishing on $U$.

\section{Irreducible holomorphic symplectic manifolds and antisymplectic involutions}\label{s-2}

\subsection{Lattices}\label{ss-2-1}

	\begin{dfn}\label{d-2-1}
		A lattice $L$ is a finitely generated free abelian group equipped with an integral non-degenerate symmetric bilinear form $( \cdot , \cdot ) : L \times L \to \mathbb{Z}$.
	\end{dfn}
	
	The rank of a lattice $L$ is denoted by $\rk(L)$.
	For $K= \mathbb{Q}, \mathbb{R}, \text{or } \mathbb{C}$, we set $L_K= L \otimes_{\mathbb{Z}} K$.
	The signature of $L$ is denoted by $\sign(L)$.
	The dual lattice is denoted by $L^{\vee} = \operatorname{Hom}(L, \mathbb{Z})$.
	It can be identified with the subgroup of $L_{\mathbb{Q}}$ defined by $\{ x \in L_{\mathbb{Q}} ; (x, y) \in \mathbb{Z} \text{ for all } y \in L \}$.
	Since the bilinear form $( \cdot , \cdot )$ is non-degenerate, there is an injection $L \to L^{\vee}$.
	The quotient $A_L = L^{\vee}/L$ is called the discriminant group of $L$.
	
	A lattice $L$ is called unimodular if $A_L =\{ 0 \}$.
	For each $l \in L$, set $l^2 = (l,l) \in \mathbb{Z}$.
	A lattice $L$ is called even if $l^2$ is an even number for any $l \in L$.
	We denote by $U$ the rank $2$ even unimodular lattice of signature $(1,1)$ with Gram matrix  
	$ \begin{pmatrix}
	0 & 1 \\
	1 & 0 \\
	\end{pmatrix}$,
	and denote by $E_8$ the {\it negative} definite even unimodular lattice associated with the Dynkin diagram $E_8$.
	We set 
	$$
		L_{K3} = E_8^{\oplus 2} \oplus U^{\oplus 3} \quad \text{and} \quad L_2 = L_{K3} \oplus \mathbb{Z} e,
	$$
	where $e^2 =-2$ and $(e, L_{K3})=0$.
	Let $k \in \mathbb{Z}$ and let $( \cdot , \cdot )$ be the bilinear form of a lattice $L$.
	We denote by $L(k)$ the lattice which is the free abelian group $L$ equipped with the bilinear form $k( \cdot , \cdot )$. 
	
	The isometry group of $L$ is denoted by $O(L)$.
	A sublattice $S \subset L$ is primitive if the quotient $L/S$ is a free abelian group.
	A lattice $L$ is hyperbolic if $\sign(L) = (1, \rk(L)-1)$.
	A lattice $L$ is 2-elementary if there exists a non-negative integer $l(L) \in \mathbb{Z}_{\geqq 0}$ such that $A_L$ is isomorphic to $(\mathbb{Z}/2\mathbb{Z})^{l(L)}$.
	
	For a lattice $L$ and $l \in L$, the reflection $s_l : L_{\mathbb{R}} \to L_{\mathbb{R}}$ is defined by
	$$
		s_l (x) = x - \frac{2(x,l)}{(l,l)}l .
	$$
	Any isometry $g$ can be expressed as the product of reflections
	$$
		g = s_{v_1} \dots s_{v_m}, 
	$$
	where $v_1, \dots v_m$ are elements of $L_{\mathbb{R}}$.
	We define the real spinor norm $sn_{\mathbb{R}}(g)$ by
	$$
		sn_{\mathbb{R}}(g) = \left\{ 
				\begin{split}
					& +1 &\quad  \text{if} \quad \left( -(v_1)^2 \right) \dots \left( -(v_m)^2 \right) >0, \\
					& -1 & \quad \text{if} \quad \left( -(v_1)^2 \right) \dots \left( -(v_m)^2 \right) <0.
				\end{split} \right.
	$$
	This is independent of the choice of $v_1, \dots v_m$.
	We define a subgroup $O^+(L)$ of $O(L)$ by
	$$
		O^+(L) = \left\{ g \in O(L) ; sn_{\mathbb{R}}(g) =+1 \right\}.
	$$
	
	Let $N$ be a lattice of signature $(2, n)$.
	We set
	$$
		\Omega_{N} = \{ [\eta] \in \mathbb{P}(N_{\mathbb{C}}) ; (\eta, \eta)=0, (\eta, \bar{\eta}) >0 \}.
	$$
	Then $\Omega_{N}$ is a complex manifold
	consisting of two connected components $\Omega_1$ and $\Omega_2$, both of which are isomorphic to a bounded symmetric domain of type IV of dimension $n$.
	For $v \in L$ with $(v,v) \neq 0$, we have
	$$
		s_v(\Omega_1) =  \left\{ 
				\begin{split}
					& \Omega_1 &\quad  \text{if} \quad (v,v) < 0  \\
					& \Omega_2 & \quad \text{if} \quad (v,v) > 0
				\end{split} \right.
	$$
	and we have
	\begin{align}\label{al-2-A}
		O^+(N) = \{ g \in O(N) ; \text{$g$ preserves $\Omega_1$} \}.
	\end{align}
	Let $\Gamma$ be a finite index subgroup of $O^+(\Lambda)$.
	For each $i=1,2$, $\Gamma$ acts on $\Omega_i$ projectively.
	We define an orthogonal modular variety $\mathcal{M}$ by 
	$$
		\mathcal{M}=\Omega_i/\Gamma.
	$$ 
	By \cite[Theorems 10.4 and 10.11]{MR216035}, $\mathcal{M}$ has a compactification $\mathcal{M}^*$, called the Baily-Borel compactification,
	such that $\mathcal{M}^*$ is an irreducible normal projective variety of dimension $n$
	and such that the boundary $\mathcal{M}^* \setminus \mathcal{M}$ is of codimension $\geqq 2$ if $n \geqq 3$. 
	
	Let $\Lambda$ be a lattice of signature $(3, n)$.
	By \cite[Lemma 4.1]{MR2964480}, an isometry $g \in O(\Lambda)$ is of real spinor norm $+1$ if and only if it acts on $H^2(\tilde{\mathscr{C}}_{\Lambda}, \mathbb{Z}) \cong \mathbb{Z}$ by $+1$,
	where $\tilde{\mathscr{C}}_{\Lambda} = \{ x \in \Lambda_{\mathbb{R}} ; x^2 > 0 \}$.
	A generator of $H^2(\tilde{\mathscr{C}}_{\Lambda}, \mathbb{Z}) \cong \mathbb{Z}$ is called an orientation class of $\tilde{\mathscr{C}}_{\Lambda}$.
	Let $h \in \Lambda$ be an element with $(h, h)>0$.
	Since the signature of $h^{\perp} \in \Lambda$ is $(2, n)$, $\Omega_{h^{\perp}} = \Omega_{\Lambda} \cap h^{\perp}$ consists of two connected components.
	If an isometry $g \in O(\Lambda)$ is of real spinor norm $+1$ and $g(h) =h$, it preserves connected components of $\Omega_{h^{\perp}}$.

\subsection{Irreducible holomorphic symplectic manifolds and antisymplectic involutions}\label{ss-2-2}

	\begin{dfn}\label{d-2-2}
		A simply-connected compact \K manifold $X$ is an irreducible holomorphic symplectic manifold if there exists an everywhere non-degenerate holomorphic 2-form $\eta$ such that $H^0(X, \Omega_X^2) = \mathbb{C} \eta$.
	\end{dfn}
	
	\begin{dfn}\label{d-2-3}
		A compact Riemannian manifold $(M, g)$ of dimension $4n$ is a hyperk\"ahler manifold if its holonomy group is equal to $Sp(n)$.
	\end{dfn}
	
	If $(M, g)$ is hyperk\"ahler, then there exist three complex structures $I$, $J$, and $K$ on $M$ such that $IJ=-JI=K$ and that $(M,g,I)$, $(M,g,J)$, $(M,g,K)$ are all irreducible holomorphic symplectic manifolds \cite[Proposition 23.3]{MR1963559}. 
	
	On the other hand, let $X$ be an irreducible holomorphic symplectic manifold,
	and let $\alpha \in H^2(X, \mathbb{R})$ be a \K class of $X$.
	By Yau \cite{MR480350}, there exists a unique Ricci-flat \K form $\omega_{X,0}$ such that $[\omega_{X,0}]=\alpha$ in the cohomology $H^2(X, \mathbb{R})$.
	If $M$ is the real manifold underlying $X$ and $g$ is the Ricci-flat Riemannian metric corresponding to $\omega_{X,0}$,
	then $(M,g)$ is a hyperk\"ahler manifold \cite[Proposition 5.11]{MR1963559}.
	
	Let $X$ be an irreducible holomorphic symplectic manifold of dimension $2n$.
	The $k$-th Betti number of $X$ is denoted by $b_k(X)$.
	By \cite[Proposition 23.14 and Remark 23.15]{MR1963559}, there exists a unique primitive integral quadric form $q_X$ on $H^2(X, \mathbb{Z})$ of signature $(3, b_2(X)-3)$ such that it satisfies the following property.
	There exists a positive rational number $c_X \in \mathbb{Q}_{\geqq 0}$ such that $q_X(\alpha)^n = c_X \int_X \alpha^{2n}$ for any $\alpha \in H^2(X, \mathbb{Z})$. 
	If $b_2(X)=6$, we also require that $q_X(\omega)>0$ for any \K class $\omega$. 
	The quadric form $q_X$ is called the Beauville-Bogomolov-Fujiki form.
	Let $( \cdot , \cdot ) : H^2(X, \mathbb{Z}) \times H^2(X, \mathbb{Z}) \to \mathbb{Z}$ be the integral bilinear form corresponding to $q_X$.
	
	\begin{exa}\label{e-2-4}
		A $2$-dimensional irreducible holomorphic symplectic manifold is a K3 surface.
		In this case, the Beauville-Bogomolov-Fujiki form is the cup product.
	\end{exa}
	
	\begin{exa}\label{e-2-5}
		The Hilbert scheme $Y^{[n]}$ of length $n$ zero-dimensional subschemes of a K3 surface $Y$ is an irreducible holomorphic symplectic manifold of dimension $2n$. \textup{(\cite[Th\'eor\`eme 3.]{MR730926})} \par
		If the $n$-th symmetric product of a K3 surface $Y$ is denoted by $Y^{(n)}$,
		the Hilbert scheme $Y^{[n]}$ is defined by a crepant resolution of $Y^{(n)}$.
		Such a crepant resolution is unique by \cite[Theorem (2.2)]{MR2069119}.
		Let $E$ be its exceptional divisor and its Poincar\'e dual is denoted by $[E]$.
		Let $\tau : Y^{[n]} \to Y^{(n)}$ be the crepant resolution, $\pi : Y^n \to Y^{(n)}$ be the projection, and $p_k : Y^n \to Y$ be the $k$-th projection.
		We define  an injection $i : H^2(Y, \mathbb{Z}) \to H^2(Y^{[n]}, \mathbb{Z})$ by $i(\alpha) = \tau^*(\beta)$, where $\beta \in H^2(Y^{(n)}, \mathbb{Z})$ is determined by 
		$$
			\pi^* ( \beta ) = \sum_k p_k^* ( \alpha ).
		$$
		By \textup{\cite[Proposition~6]{MR730926}}, $i : H^2(Y, \mathbb{Z}) \to H^2(Y^{[n]}, \mathbb{Z})$ preserves the Beauville-Bogomolov-Fujiki forms and 
		\begin{align}\label{f-2-1}
			H^2(Y^{[n]}, \mathbb{Z}) \cong i \left( H^2(Y, \mathbb{Z}) \right) \oplus \mathbb{Z} \varepsilon,
		\end{align} 
		where $\varepsilon \in H^2(Y^{[n]}, \mathbb{Z})$ is the cohomology class such that $2\epsilon = [E]$.
	\end{exa}
	
	\begin{dfn}\label{d-2-6}
		An irreducible holomorphic symplectic manifold $X$ is of $K3^{[n]}$-type if $X$ is deformation equivalent to the Hilbert scheme of $n$-points of a K3 surface.
	\end{dfn}
	
	Let $X$ be an irreducible holomorphic symplectic manifold and $\eta$ a holomorphic symplectic 2-form on $X$.
	For any holomorphic involution $\iota : X \to X$, we have $\iota^* \eta =\eta$ or $-\eta$.
	A holomorphic involution $\iota : X \to X$ which satisfies $\iota^* \eta =-\eta$ is called antisymplectic.
	
	\begin{exa}\label{e-2-8}
		Let $Y$ be a K3 surface and let $\sigma : Y \to Y$ be an antisymplectic involution of $Y$.
		The involution $\sigma$ induces a holomorphic involution
		$
			\sigma^{[2]} : Y^{[2]} \to Y^{[2]},
		$
		and $\sigma^{[2]}$ is an antisymplectic involution.
		The involution $\sigma^{[2]}$ on $Y^{[2]}$ is called the natural involution.
	\end{exa}
	
	\begin{exa}\label{e-2-9}
		We recall the example of antisymplectic involution in \cite[proof of Corollary 2.11]{ohashi2013non} and \cite[Example 9.12]{MR3519981}. 
		Let $C \subset \mathbb{P}^2$ be a smooth sextic curve.
		The double covering of $\mathbb{P}^2$ branched over $C$ is denoted by $\pi : Y \to \mathbb{P}^2$,
		and the covering involution is denoted by $\sigma : Y \to Y$.
		Then $Y$ is a K3 surface and $\sigma$ is antisymplectic.
		By Example \ref{e-2-8}, $(Y,\sigma)$ induces a manifold of $K3^{[2]}$-type $(Y^{[2]}, \sigma^{[2]})$ with antisymplectic involution.
		Its fixed locus is isomorphic to $C^{(2)} \sqcup (Y/\sigma) = C^{(2)} \sqcup \mathbb{P}^2$.  
		Let $f : Y^{[2]} \dashrightarrow \elm_{Y/\sigma}(Y^{[2]})$ be the Mukai flop of $Y^{[2]}$ along $Y/\sigma =\mathbb{P}^2$. (See \cite[Example 21.7]{MR1963559}.)
		Set 
		$
			\elm_{Y/\sigma}(\sigma^{[2]}) = f \circ \sigma^{[2]} \circ f^{-1}.
		$
		Then $\elm_{Y/\sigma}(Y^{[2]})$ is a manifold of $K3^{[2]}$-type.
		By  \cite[proof of Corollary 2.11]{ohashi2013non}, $\elm_{Y/\sigma}(\sigma^{[2]}) : \elm_{Y/\sigma}(Y^{[2]}) \to \elm_{Y/\sigma}(Y^{[2]})$ is biregular and is an antisymplectic involution.
	\end{exa}
	
	The invariant subspace of $H^2(X, \mathbb{Z})$ is defined by
	$
		H^2(X, \mathbb{Z})^{\iota} = \left\{ \alpha \in H^2(X, \mathbb{Z}) ; \iota^* \alpha = \alpha \right\}.
	$
	
	\begin{lem}\label{l-2-10}
		If $\iota$ is antisymplectic, then the following hold:
		\begin{enumerate}[ label= \rm{(\arabic*)} ]
			\item $(\eta, H^2(X, \mathbb{Z})^{\iota}) =0$.
			\item $H^2(X, \mathbb{Z})^{\iota} \subset H^{1,1}(X, \mathbb{Z})$.
			\item $H^2(X, \mathbb{Z})^{\iota}$ is hyperbolic. Namely, $\sign(q_X|_{H^2(X, \mathbb{Z})^{\iota}}) = (1, \rk H^2(X, \mathbb{Z})^{\iota} -1).$
			\item $X$ is projective.
		\end{enumerate}
	\end{lem}
	
	\begin{proof}
		See \cite[Proposition 4.4]{MR3519981}.
	\end{proof}
	
	In what follows, involutions on a $K3^{[2]}$-type manifold always imply antisymplectic ones.
	
	\begin{dfn}\label{d-2-11}
		Let $\mathscr{X}$and $S$ be complex manifolds, let $f : \mathscr{X} \to S$ be a surjective proper holomorphic submersion, and let $\iota : \X \to \X$ be a holomorphic involution.
		Then $f : (\X, \iota) \to S$ is called a family of $K3^{[2]}$-type manifolds with involution if it satisfies the following three conditions:
		\begin{enumerate}[ label= \rm{(\arabic*)} ]
			\item For each $s \in S$, $X_s = f^{-1}(s)$ is a manifold of $K3^{[2]}$-type.
			\item $f \circ \iota = f$.
			\item $\iota : \X \to \X$ induces an antisymplectic involution $\iota_s : X_s \to X_s$ for all $s \in S$.
		\end{enumerate}
	\end{dfn}
	
	Let $(X, \iota)$ be a manifold of $K3^{[2]}$-type with involution,
	and let $\pi : \X \to (\operatorname{Def}(X), 0)$ be the Kuranishi family of $X$ with $\pi^{-1}(0)=X$.
	Since $\pi : \X \to \operatorname{Def}(X)$ is a universal family (\cite[\S 22.1]{MR1963559}), there exists a holomorphic involution $I : \X \to \X$ and $J : (\operatorname{Def}(X), 0) \to (\operatorname{Def}(X), 0)$ such that $I|_{X} = \iota$ 
	and the following diagram commutes
	$$
		\xymatrix{
			\X  \ar[r]^{I}  \ar[d]_{\pi} & \X  \ar[d]^{\pi}   \\
			(\operatorname{Def}(X), 0) \ar[r]_{J} & (\operatorname{Def}(X), 0) 
		}
	$$
	The fixed locus of $J$ is called the local deformation space $\operatorname{Def}(X, \iota)$ of $(X, \iota)$.
	
	\begin{lem}\label{l-2-12}
		Let $(X, \iota)$ be a manifold of $K3^{[2]}$-type with involution,
		and let $\eta$ be a holomorphic symplectic 2-form on $X$.
		Set 
			$$
				t=\operatorname{Tr}(\iota^*|_{H^{1,1}(X)}).
			$$
		Then the following holds.
		\begin{enumerate}[ label= \rm{(\arabic*)} ]
			\item The fixed locus $X^{\iota}$ is a smooth Lagrangian submanifold of $X$. 
			Namely, $X^{\iota}$ is a smooth complex surface such that $\eta|_{X^{\iota}} =0$.
			\item $\int_{X^{\iota}} c_1(X^{\iota})^2 = t^2-1$, \quad $\chi(\mathcal{O}_{X^{\iota}}) = \frac{t^2+7}{8}$, \quad and  $\int_{X^{\iota}} c_2(X^{\iota}) = \frac{t^2+23}{2}$.
			\item The local deformation space $\operatorname{Def}(X, \iota)$ is smooth of dimension $\frac{21-t}{2}$.
			\item $t$ is an odd number with $-19 \leqq t \leqq 21$.
		\end{enumerate}
	\end{lem}
	
	\begin{proof}
		See \cite[THEOREMS~1 and 2]{MR2805992}.
	\end{proof}
	
	\begin{exa}\label{e-2-13}
		By the construction of $i : H^2(Y, \mathbb{Z}) \to H^2(Y^{[2]}, \mathbb{Z})$ and $E$ in Example \ref{e-2-5}, the antisymplectic involution $\sigma^{[2]} : Y^{[2]} \to Y^{[2]}$ satisfies $(\sigma^{[2]})^* \circ i = i \circ \sigma^*$ and $(\sigma^{[2]})^*[E] =[E]$.
		Therefore, we have
		\begin{align*}
			H^2(Y^{[2]}, \mathbb{Z})^{\sigma^{[2]}} = i \left( H^2(Y, \mathbb{Z})^{\sigma} \right) \oplus \mathbb{Z} \varepsilon.
		\end{align*} 
		By \cite[Example 5.1]{MR3519981}, we have $\operatorname{Def}(Y, \sigma)=\operatorname{Def}(Y^{[2]}, \sigma^{[2]})$ 
		and any deformation of $(Y^{[2]}, \sigma^{[2]})$ is induced from a deformation of $(Y, \sigma)$. 
	\end{exa}

\subsection{K\"ahler-type chambers}\label{ss-2-3}

	Following Joumaah \cite{MR3519981}, we recall the deformation type, the moduli space, and the period map for $K3^{[2]}$-type manifolds with antisymplectic involution.

	Let $X_1, X_2$ be irreducible holomorphic symplectic manifolds.
	Recall that a parallel-transport operator $f : H^2(X_1, \mathbb{Z}) \to H^2(X_2, \mathbb{Z})$ is an isomorphism such that
	there exist a family $p : \mathcal{X} \to B$ of irreducible holomorphic symplectic manifolds over a possibly reducible analytic base $B$, 
	two points $b_1, b_2 \in B$, and a continuous path $\gamma : [0, 1] \to B$ with $\gamma(0)=b_1, \gamma(1)=b_2$ such that $p^{-1}(b_i) \cong X_i$ $(i=1,2)$ and that the parallel-transport in the local system $R^2p_*\mathbb{Z}$ induces $f : H^2(X_1, \mathbb{Z}) \to H^2(X_2, \mathbb{Z})$.

	\begin{dfn}\label{d3-2-1}
		Let $X$ be an irreducible holomorphic symplectic manifold.
		A parallel-transport operator $g : H^2(X, \mathbb{Z}) \to H^2(X, \mathbb{Z})$ is called a monodromy operator.
		The subgroup $\mon(X)$ of $O(H^2(X, \mathbb{Z}))$ consisting of all monodromy operators of $X$ is called the monodromy group.
	\end{dfn}
	
	Let $X$ be a manifold of $K3^{[2]}$-type and $\alpha : H^2(X, \mathbb{Z}) \to L_2$ be an isometry.
	The pair $(X, \alpha)$ is called a marked manifold of $K3^{[2]}$-type.
	Let $\mon(L_2)$ be the subgroup of the isometry group $O(L_2)$ defined by 
	$$
		\mon(L_2) = \alpha \circ \mon(X) \circ \alpha^{-1}.
	$$
	By \cite[Theorem 9.1]{MR2964480}, the group $\mon(L_2)$ is a normal subgroup of $O(L_2)$ and is independent of the choice of $(X,\alpha)$.
	By \cite[Lemma 9.2]{MR2964480}, we have $\mon(L_2) = O^+(L_2)$.
	
	Similarly, we can define the monodromy group $\mon(L_{K3})$ and by \cite[Theorem A]{MR0849050} we have $\mon(L_{K3}) = O^+(L_{K3})$.

	\begin{dfn}\label{d-2-6}
		Let $M$ be a sublattice of $L_2$ and $\iota_M \in \mon(L_2)$ an involution.
		The pair $(M, \iota_M)$ is an admissible sublattice of $L_2$ if $M$ is hyperbolic and the invariant sublattice $(L_2)^{\iota_M}$ of $\iota_M$ is equal to $M$.
	\end{dfn}

	Let $\mathfrak{M}_{L_2}$ be the moduli space of marked manifolds of $K3^{[2]}$-type constructed in \cite[Definition 25.4]{MR1963559}.
	We fix a connected component $\mathfrak{M}_{L_2}^{\circ}$ of $\mathfrak{M}_{L_2}$.
	
	\begin{dfn}\label{d-2-7}
		Let $(M, \iota_M)$ be an admissible sublattice of $L_2$ and let $(X, \iota)$ be a manifold of $K3^{[2]}$-type with involution.
		An isometry $\alpha : H^2(X, \mathbb{Z}) \to L_2$ is called an ($M$-)admissible marking of $(X,\iota)$ if $(X,\alpha) \in \mathfrak{M}_{L_2}^{\circ}$ and $\alpha \circ \iota^* = \iota_M \circ \alpha$.
		Moreover $\iota$ is of type $M$ if there exists an $M$-admissible marking of $(X, \iota)$. 
	\end{dfn} 
	
	Let $(M, \iota_M)$ be an admissible sublattice of $L_2$.
	
	\begin{dfn}\label{d-2-8}
		Define a set $\Delta(M)$ by
		$$
			\Delta(M) = \left\{ \delta \in M ; \delta^2 =-2, \text{or } \delta^2=-10,  (\delta, L_2)=2\mathbb{Z} \right\}.
		$$ 
	\end{dfn}

	We set $\tilde{\mathscr{C}}_M = \{ x \in M_{\mathbb{R}} ; x^2>0 \}$.
	For $\delta \in \Delta(M)$, we define $\delta^{\perp}=\{ x \in \tilde{\mathscr{C}}_M : (x, \delta)=0 \}$.
	
	\begin{dfn}\label{d-2-9}
		A connected component of $\tilde{\mathscr{C}}_M \setminus \bigcup_{\delta \in \Delta(M)} \delta^{\perp}$ is called a K\"ahler-type chamber of $M$.
		The set of all K\"ahler-type chambers of $M$ is denoted by $\ktm$.
	\end{dfn}
	
	We set
	$$
		\Gamma(M) =\{ \sigma \in \mon(L_2) ; \sigma \circ \iota_M = \iota_M \circ \sigma \},
	$$
	and we define a subgroup $\Gamma_M$ of $O(M)$ by
	$$
		\Gamma_M = \{ \sigma|_M \in O(M) ; \sigma \in \Gamma(M) \}.
	$$
	The group $\Gamma_M$ acts on $\ktm$.
	
	Let $(X, \iota)$ be a manifold of $K3^{[2]}$-type with involution of type $M$, and let $\alpha : H^2(X, \mathbb{Z}) \to L_2$ be an $M$-admissible marking of $(X, \iota)$.
	The $\iota$-invariant \K cone of $X$ is defined by 
	$$
		\mathcal{K}_X^{\iota} = \{ \omega \in H^{1,1}(X, \mathbb{Z}) ; \text{ $\omega$ is a \K class, }  \iota^*\omega = \omega \}.
	$$
	We denote by $\rho(X, \iota, \alpha)$ the K\"ahler-type chamber containing $\alpha(\mathcal{K}_X^{\iota})$.
	Let $\alpha, \beta : H^2(X, \mathbb{Z})  \to  L_2$ are two $M$-admissible markings of $(X, \iota)$. 
	Since $(X, \alpha), (X, \beta) \in \mathfrak{M}_{L_2}^{\circ}$, we have $\alpha \circ \beta^{-1} \in \mon(L_2)$. 
	Moreover, $\alpha \circ \beta^{-1} \in \Gamma(M)$ and hence
	$$
		[\rho(X, \iota, \alpha)] = [\rho(X, \iota, \beta)]
	$$
	in $\ktm/\Gamma_M$.
	Therefore the class $[\rho(X, \iota, \alpha)] \in \ktm/\Gamma_M$ is independent of the choice of $\alpha$,
	and we denote it by $\rho(X, \iota)$.
	
	\begin{lem}\label{l-2-3}
		The map $\rho$ induces a bijection from the set of deformation types of $K3^{[2]}$-type manifolds with antisymplectic involution of type $M$ to $\ktm/\Gamma_M$.
	\end{lem}
	
	\begin{proof}
		See \cite[Theorem 9.11]{MR3519981}.
	\end{proof}
	
	Let $\mathcal{K} \in \ktm$ and fix $h \in \mathcal{K} \cap M$.
	Then $h^{\perp} = h^{\perp} \cap L_2$ is a lattice of signature $(2, 20)$ containing $M^{\perp}$, and $\Omega_{h^{\perp}}$ consists of two connected components.
	By \cite[(4.1)]{MR2964480}, the connected component $\mathfrak{M}_{L_2}^{\circ}$ determines the connected component $\Omega_{h^{\perp}}^+$ of $\Omega_{h^{\perp}}$.
	Namely, for any $(X, \alpha) \in \mathfrak{M}_{L_2}^{\circ}$ and $p = [\sigma] \in \Omega_{h^{\perp}}^+$, the orientation of $\tilde{\mathscr{C}}_{X} = \{ x \in H^2(X, \mathbb{R}) ; q_X(x)>0 \}$ determined by the \K cone of $X$ is compatible with 
	the orientation of $\tilde{\mathscr{C}_{L_2}} = \{ x \in L_{2, \mathbb{R}} ; x^2>0 \}$ determined by the real 3-dimensional vector space $W = \operatorname{Span}_{\mathbb{R}} \{ \operatorname{Re} \sigma, \operatorname{Im} \sigma, h \}$ associated to the basis $\{ \operatorname{Re} \sigma, \operatorname{Im} \sigma, h \}$ via the isomorphism $\tilde{\mathscr{C}}_X \cong \tilde{\mathscr{C}_{L_2}}$ induced from the marking $\alpha$.

	Let $\Omega^+_{M^{\perp}}$ be the connected component of $\Omega_{M^{\perp}}$ which satisfies $\Omega^+_{M^{\perp}} \subset \Omega^+_{h^{\perp}}$.
	Set 
	$$
		\mathfrak{M}_{M^{\perp}, \mathcal{K}} =\{ (X,\alpha) \in \mathfrak{M}^{\circ}_{L^2} ; \alpha(H^{2,0}(X)) \in \Omega^+_{M^{\perp}} \text{ and } \mathcal{K} \cap \alpha(\mathcal{K}_X) \neq \emptyset \},
	$$
	where $\mathcal{K}_X$ is the \K cone of the $K3^{[2]}$-type manifold $X$.
	We define a map $P_{\mathcal{K}} : \mathfrak{M}_{M^{\perp}, \mathcal{K}} \to \Omega^+_{M^{\perp}}$ by $P_{\mathcal{K}}(X, \alpha) = \alpha(H^{2,0}(X))$.
	
	We denote by $\tilde{\mathcal{M}}_{M, \mathcal{K}}$ the set of isomorphism classes of $K3^{[2]}$-type manifolds $(X, \iota)$ with involution of type $M$ such that $\rho(X, \iota) = [\mathcal{K}]$ in $\ktm/\Gamma_M$.
	
	\begin{dfn}\label{d-2-10}
		Let $(X, \iota) \in \tilde{\mathcal{M}}_{M, \mathcal{K}}$. 
		An isometry $\alpha : H^2(X, \mathbb{Z}) \to L_2$ is admissible for $(M, \mathcal{K})$ if $\alpha$ is an $M$-admissible marking and $\alpha(\mathcal{K}_X^{\iota}) \subset \mathcal{K}$,
		where $\mathcal{K}_X^{\iota}$ is the $\iota$-invariant \K cone of $X$.
	\end{dfn}
	
	For each $(X, \iota) \in \tilde{\mathcal{M}}_{M, \mathcal{K}}$, there exists an admissible marking for $(M, \mathcal{K})$.
	Moreover, if $\alpha, \beta : H^2(X, \mathbb{Z}) \to L_2$ are two admissible markings for $(M, \mathcal{K})$, 
	then we have $\alpha \circ \beta^{-1} \in \Gamma(\mathcal{K})$,
	where
	$$
		\Gamma(\mathcal{K}) =\{ \sigma \in \mon(L_2) ; \sigma \circ \iota_M = \iota_M \circ \sigma  \text{ and } \sigma(\mathcal{K})= \mathcal{K} \} =\{ \sigma \in \Gamma(M) ; \sigma(\mathcal{K})=\mathcal{K} \}.
	$$
	
	Let
	$$
		\Gamma_{M^{\perp}, \mathcal{K}} = \{ \sigma|_{M^{\perp}} \in O(M^{\perp}) ; \sigma \in \Gamma(\mathcal{K}) \}.
	$$ 
	For $g \in \Gamma_{M^{\perp}, \mathcal{K}}$, there exists $\sigma \in \Gamma(\mathcal{K})$ such that $\sigma|_{M^{\perp}} =g$.
	Since $\Gamma(\mathcal{K}) \subset \mon(L_2) = O^+(L_2)$, we have $sn_{\mathbb{R}}(\sigma) =+1$.
	On the other hand, since $(\sigma|_M)(\mathcal{K}) =\mathcal{K}$, we have $sn_{\mathbb{R}}(\sigma|_M) =+1$.
	Therefore, $sn_{\mathbb{R}}(g) = sn_{\mathbb{R}}(\sigma|_{M^{\perp}}) =+1$ and $g \in O^+(M^{\perp})$.
	Thus, $\Gamma_{M^{\perp}, \mathcal{K}}$ is contained in $O^+(M^{\perp})$.
	By \cite[Proposition 10.2]{MR3519981}, $\Gamma_{M^{\perp}, \mathcal{K}}$ is a finite index subgroup of $O^+(M^{\perp})$.
	By (\ref{al-2-A}), $O^+(M^{\perp})$ preserves $\Omega^+_{M^{\perp}}$.
	Therefore we obtain an orthogonal modular variety
	$$
		\mathcal{M}_{M, \mathcal{K}} = \Omega^+_{M^{\perp}}/\Gamma_{M^{\perp}, \mathcal{K}} .
	$$
	
	We define the period map $P_{M, \mathcal{K}} : \tilde{\mathcal{M}}_{M, \mathcal{K}} \to \mathcal{M}_{M, \mathcal{K}}$ by
	$$
		P_{M, \mathcal{K}}(X, \iota) = [\alpha(H^{2,0}(X))],
	$$
	where $\alpha : H^2(X, \mathbb{Z}) \to L_2$ is an admissible marking for $(M, \mathcal{K})$.
	
	Let 
	$$
		\Delta(M^{\perp}) =\{ \delta \in M^{\perp} ; \delta^2=-2, \text{or } \delta^2=-10, (\delta, L_2)=2\mathbb{Z} \},
	$$
	and set 
	$$
		\mathscr{D}_{M^{\perp}} = \bigcup_{\delta \in \Delta(M^{\perp})} H_{\delta} \subset \Omega_{M^{\perp}},
	$$
	where
	$$
		H_{\delta} = \{ x \in \Omega_{M^{\perp}} ; (x, \delta)=0 \}.
	$$
	By \cite[Lemma 7.7]{MR3519981}, $\mathscr{D}_{M^{\perp}}$ is locally finite in $\Omega_{M^{\perp}}$
	and is viewed as a reduced divisor on $\Omega_{M^{\perp}}$.
	Set $\bar{\mathscr{D}}_{M^{\perp}}= \mathscr{D}_{M^{\perp}}/\Gamma_{M^{\perp}, \mathcal{K}}$.
	Then $\bar{\mathscr{D}}_{M^{\perp}}$ is a reduced divisor on $\mathcal{M}_{M, \mathcal{K}}$.
	We set 
	$$
		\mathcal{M}^{\circ}_{M, \mathcal{K}} = \mathcal{M}_{M, \mathcal{K}} \setminus \bar{\mathscr{D}}_{M^{\perp}} \quad \text{  and  } \quad \Omega_{M^{\perp}}^{\circ} = \Omega^+_{M^{\perp}} \setminus \mathscr{D}_{M^{\perp}}.
	$$
	
	\begin{lem}\label{l-2-4}
		The image of $P_{M, \mathcal{K}}$ is $\mathcal{M}^{\circ}_{M, \mathcal{K}}$.
	\end{lem}
	
	\begin{proof}
		See \cite[Lemma 9.5 and Proposition 9.9]{MR3519981}.
	\end{proof}
	
	The period map $P_{M, \mathcal{K}} : \tilde{\mathcal{M}}_{M, \mathcal{K}} \to \mathcal{M}^{\circ}_{M, \mathcal{K}}$ is not injective but generically injective.
	
	\begin{thm}\label{l-2-1-A}
		There exists a $\Gamma_{M^{\perp}, \mathcal{K}}$-invariant effective reduced divisor $\mathscr{D}_{\mathcal{K}}$ such that 
		$$
			P_{M, \mathcal{K}} : P_{M, \mathcal{K}}^{-1}( \mathcal{M}^{\circ}_{M, \mathcal{K}} \setminus \bar{\mathscr{D}}_{\mathcal{K}} ) \to \mathcal{M}^{\circ}_{M, \mathcal{K}} \setminus \bar{\mathscr{D}}_{\mathcal{K}}
		$$
		is bijective,
		where $\bar{\mathscr{D}}_{\mathcal{K}} = \mathscr{D}_{\mathcal{K}} / \Gamma_{M^{\perp}, \mathcal{K}}$.
	\end{thm}
	
	\begin{proof}
		See \cite[Theorem 10.5]{MR3519981}.
	\end{proof}
	
	Let $f : (\mathscr{X}, \iota) \to S$ be a family of $K3^{[2]}$-type manifolds with involution of type $(M, \mathcal{K})$.
	We define the period map $P_{M, \mathcal{K}} : S \to \mathcal{M}_{M, \mathcal{K}}$ by
	$$
		P_{M, \mathcal{K}}(s) = P_{M, \mathcal{K}}(X_s, \iota_s), \qquad (s \in S). 
	$$
	Since $f_*\Omega^2_{\xs}$ is a holomorphic vector subbundle of the flat bundle $R^2f_*\mathbb{C} \otimes \mathcal{O}_{S}$, $P_{M, \mathcal{K}}$ is holomorphic (\cite[22.3]{MR1963559}).
	
	Let $(X, \iota), (X', \iota') \in \tilde{\mathcal{M}}_{M, \mathcal{K}}$.
	We call $(X, \iota)$ and $(X', \iota')$ inseparable if their universal deformations $\pi : (\mathscr{X}, \iota) \to \operatorname{Def}(X, \iota)$ and $\pi' : (\mathscr{X}', \iota') \to \operatorname{Def}(X', \iota')$ contain isomorphic fibers,
	where $\operatorname{Def}(X, \iota)$ and $\operatorname{Def}(X', \iota')$ are viewed as germs.
	
	\begin{lem}\label{l-2-5}
		If $(X, \iota), (X', \iota') \in \tilde{\mathcal{M}}_{M, \mathcal{K}}$ satisfy $P_{M, \mathcal{K}}(X, \iota) = P_{M, \mathcal{K}}(X', \iota')$, then $(X, \iota)$ and $(X', \iota')$ are inseparable.
	\end{lem}
	
	\begin{proof}
		See \cite[Proposition 10.7]{MR3519981}.
	\end{proof}

\subsection{Relations of orthogonal modular varieties}\label{ss-2-4}
	
	Let $M_0$ be a primitive hyperbolic 2-elementary sublattice of $L_{K3}$.
	Since $L_{K3}$ is unimodular and since $M_0$ is 2-elementary, the involution
	$$
		M_0 \oplus M_0^{\perp} \to M_0 \oplus M_0^{\perp}, \quad (m,n) \mapsto (m,-n) 
	$$
	extends uniquely to an involution $\iota_{M_0} \in O(L_{K3})$ by \cite[Corollary 1.5.2]{Nikulin1980IntegralSB}.
	
	Let $Y$ be a K3 surface and $\sigma : Y \to Y$ be an antisymplectic involution on $Y$.
	Set
	$$
		H^2(Y, \mathbb{Z})^{\sigma} = \left\{ x \in H^2(Y, \mathbb{Z}) ; \sigma^*x = x \right\}.
	$$
	Let $\alpha : H^2(Y, \mathbb{Z}) \to L_{K3}$ be an isometry.
	We call the pair $(Y, \alpha)$ a 2-elementary K3 surface of type $M_0$ if the restriction of $\alpha$ is an isometry from $H^2(Y, \mathbb{Z})^{\sigma}$ to $M_0$.
	
	Since $\sign(M_0^{\perp}) =(2, \rk M_0^{\perp}-2)$, $\Omega_{M_0^{\perp}}$ consists of two connected components.
	We fix a connected component $\Omega^+_{M_0^{\perp}}$ of $\Omega_{M_0^{\perp}}$.
	By (\ref{al-2-A}), $O^+(M_0^{\perp})$ acts on $\Omega^+_{M_0^{\perp}}$ projectively.
	We obtain the orthogonal modular variety 
	$$
		\mathcal{M}_{M_0}=\Omega^+_{M_0^{\perp}}/O^+(M_0^{\perp})
	$$
	of dimension $20-\rk(M_0)$.
	
	We set
	$
		\Delta(M_0^{\perp}) =\{ d \in M_0^{\perp} ; d^2=-2 \},
	$
	and
	$$
		\mathscr{D}_{M_0^{\perp}} = \bigcup_{d \in \Delta(M_0^{\perp})} d^{\perp} \subset \Omega^+_{M_0^{\perp}}.
	$$
	By \cite[Proposition 1.9.]{MR2047658}, $\mathscr{D}_{M_0^{\perp}}$ is locally finite and is viewed as a reduced divisor on $\Omega^+_{M_0^{\perp}}$.
	Set 
	$$
		\bar{\mathscr{D}}_{M_0^{\perp}} = \mathscr{D}_{M_0^{\perp}} / O^+(M_0^{\perp}) \quad \text{ and } \quad \mathcal{M}^{\circ}_{M_0} = \mathcal{M}_{M_0} \setminus \bar{\mathscr{D}}_{M_0^{\perp}}.
	$$
	
	\begin{lem}\label{l-2-6}
		The Zariski open subset $\mathcal{M}^{\circ}_{M_0}$ is a coarse moduli space of 2-elementary K3 surfaces of type $M_0$. 
	\end{lem} 
	
	\begin{proof}
		See \cite[Theorem 1.8.]{MR2047658}
	\end{proof}
	
	Set $\Delta(M_0) =\{ d \in M_0 ; d^2=-2 \}$, and $\tilde{\mathscr{C}}_{M_0} = \{ x \in M_{0,{\mathbb{R}}} ; x^2>0 \}$.
	As before, the set of connected components of $\tilde{\mathscr{C}}_{M_0} \setminus  \bigcup_{d \in \Delta(M_0)} d^{\perp}$ is denoted by $\ktmo$.
	
	Let $\mathcal{K}_0 \in \ktmo$.
	Set 
	$$
		\Gamma(\mathcal{K}_0) = \{ \sigma \in \mon(L_{K3}); \sigma \circ \iota_{M_0} = \iota_{M_0} \circ \sigma, \text{and } \sigma(\mathcal{K}_0) =\mathcal{K}_0 \},
	$$
	and 
	$$
		\Gamma_{M_0^{\perp}, \mathcal{K}_0} =\{ \sigma|_{M_0^{\perp}} \in O(M_0^{\perp}); \sigma \in \Gamma(\mathcal{K}_0) \}.
	$$
	Since $\Gamma(\mathcal{K}_0) \subset \mon(L_{K3})$, we have $\Gamma_{M_0^{\perp}, \mathcal{K}_0} \subset O^+(M_0^{\perp})$.
	It is a finite index subgroup of $O^+(M_0^{\perp})$.
	
	\begin{lem}\label{l-2-7}
		For any $\mathcal{K}_0 \in \ktmo$, we have $\Gamma_{M_0^{\perp}, \mathcal{K}_0} = O^+(M_0^{\perp})$.
	\end{lem}
	
	\begin{proof}
		Choose $g \in O^+(M_0^{\perp})$.
		By \cite[Proposition 11.2]{MR3039773}, there exists an isometry $\sigma \in O(L_{K3})$ such that $\sigma \circ \iota_{M_0} = \iota_{M_0} \circ \sigma$ and $\sigma|_{M_0^{\perp}}=g$. 
		By \cite[Corollary 1.5.2]{Nikulin1980IntegralSB}, there exists an involution $\xi \in O(L_{K3})$ such that $\xi|_{M_0^{\perp}} = \text{id}$ and $sn_{\mathbb{R}}(\xi|_{M_0})=-1$.
		By replacing $\sigma$ with $\xi \circ \sigma$, if necessaly, we may assume that $\sigma \in O^+(L_{K3})$ and that it preserves the connected components of $\tilde{\mathscr{C}}_{M_0}$.
		
		Let $W$ be the subgroup of $\mon(L_{K3})$ generated by the reflections $s_d$ for $d \in \Delta(M_0)$.
		Note that $s_d|_{M_0^{\perp}} = \operatorname{id}_{M_0^{\perp}}$ for each $d \in \Delta(M_0)$.
		Let $\mathscr{C}_{M_0}$ be the connected component of $\tilde{\mathscr{C}}_{M_0}$ containing $\mathcal{K}_0$,
		and let $\ktmo_+$ be the set of K\"ahler-type chambers contained in $\mathscr{C}_{M_0}$.
		Then $\sigma(\mathcal{K}_0) \in \ktmo_+$ and $W$ acts on $\ktmo_+$ transitively (cf. \cite[Theorem 2.9]{MR4321993}).
		Therefore there exists an element $w \in W$ such that $(w \sigma)(\mathcal{K}_0) = \mathcal{K}_0$.
		Since $w \in W \subset \mon(L_{K3})$ and $\sigma \in O^+(L_{K3}) = \mon(L_{K3})$, we have $w \sigma \in \Gamma(\mathcal{K}_0)$ and 
		this implies $g = (w \sigma)|_{M_0^{\perp}} \in \Gamma_{M_0^{\perp}, \mathcal{K}_0}$.
	\end{proof}
		
	Recall that $L_2$ is given by $L_2 = L_{K3} \oplus \mathbb{Z}e$.
	We define a sublattice $M$ by
	$$
		M=M_0 \oplus \mathbb{Z}e ,
	$$
	and define an involution $\iota_M : L_2 \to L_2$ by
	$$
		 \iota_M(x_0 + ae)=\iota_{M_0}(x_0) +ae
	$$
	for $x_0 \in L_{K3}$ and $a \in \mathbb{Z}$.
	Then $(M, \iota_M)$ is an admissible sublattice of $L_2$.
	
	\begin{dfn}\label{d-2-11}
		Let $\mathcal{K} \in \ktm$.
		A hyperplane $H$ of $M_{\mathbb{R}}$ is a face of $\mathcal{K}$ if $H \cap \partial \mathcal{K}$ contains an open subset of $H$.
	\end{dfn}
	
	\begin{dfn}\label{d-2-12}
		A K\"ahler-type chamber $\mathcal{K} \in \ktm$ is natural if the hyperplane $M_{0, \mathbb{R}}$ is a face of $\mathcal{K}$.
	\end{dfn}
	
	\begin{lem}\label{l-2-8}
		If $\delta = d +ae \in \Delta(M)$ $(d \in M_0, a \in \mathbb{Z})$, then one of the following holds
		$$
			(1) \quad d^2 \geqq 0, \quad (2) \quad d \in \Delta(M_0), \qquad \text{or} \qquad (3) \quad \frac{d}{2} \in \Delta(M_0).
		$$
	\end{lem}

	\begin{proof}
		If $\delta^2=-2$, then $d^2 = 2a^2-2 \geqq -2$.
		Therefore we have $(1)$ $d^2 \geqq 0$ or (2) $d \in \Delta(M_0)$.\par
		Assume that $\delta^2=-10$ and $(\delta, L_2)=2\mathbb{Z}$.
		Set $d' =\frac{d}{2} \in L_{K3, \mathbb{Q}}$.
		Since $L_{K3} \subset L_2$ and $2\mathbb{Z} \supset (\delta, L_{K3}) = (d, L_{K3})$, we have $(d', L_{K3}) \subset \mathbb{Z}$.
		Since $L_{K3}$ is unimodular, we have $d' \in L_{K3}^{\vee} =L_{K3}$.
		Since $2d'=d \in M_0$ and $M_0$ is primitive, we have $d' \in M_0$.
		Therefore $d^2 = 4(d')^2 \in 8\mathbb{Z}$.
		If $d^2 \geqq 0$, then $(1)$ holds.
		So we may assume that $d^2 \leqq -8$.
		Since $\delta^2=-10$, we have $a^2 = \frac{d^2}{2}+5 \leqq 1$.
		If $a=0$, then $d^2 =-10 \notin 8\mathbb{Z}$, which is impossible.
		If $a^2=1$, then $d^2=-8$ and $(d')^2=-2$.
		Therefore, we have $(3)$ $\frac{d}{2} \in \Delta(M_0)$.
	\end{proof}
	 
	Since $M_0^{\perp}=M^{\perp}$, we may identify $O(M_0^{\perp})=O(M^{\perp})$.
	
	\begin{thm}\label{p-2-1}
		If $\mathcal{K} \in \ktm$ is natural, then $\Gamma_{M^{\perp}, \mathcal{K}} = O^+(M_0^{\perp})$.
	\end{thm}
	
	\begin{proof}
		It is clear that $\Gamma_{M^{\perp}, \mathcal{K}} \subset O^+(M^{\perp}) =O^+(M_0^{\perp})$.
		Since $e \in \Delta(M)$, either $(\mathcal{K}, e)<0$ or $(\mathcal{K}, e)>0$ holds.
		We assume that 
		\begin{align}\label{f-2-2}
			(\mathcal{K}, e)<0.
		\end{align}  \par
		\noindent
		\textbf{Step 1 } 
		Let $p : M_{\mathbb{R}} \to M_{0, \mathbb{R}}$ be the orthogonal projection.
		Since $p(\mathcal{K})$ is a connected open subset of $\tilde{\mathscr{C}}_{M_0} \setminus  \bigcup_{d \in \Delta(M_0)} d^{\perp}$,
		there exists a unique chamber $\mathcal{K}_0 \in \ktmo$ such that $p(\mathcal{K}) \subset \mathcal{K}_0$.
		By the continuity of $p$, we have $\bar{\mathcal{K}} \cap M_{0, \mathbb{R}} \subset \overline{p(\mathcal{K})}$.
		Since $\mathcal{K} \in \ktm$ is natural, $\bar{\mathcal{K}} \cap M_{0, \mathbb{R}}$ contains an open subset of $M_{0, \mathbb{R}}$, 
		and $\bar{\mathcal{K}} \cap M_{0, \mathbb{R}} \cap p(\mathcal{K}) \neq \emptyset$.
		We fix an elemtent $\omega_0 \in \bar{\mathcal{K}} \cap M_{0, \mathbb{R}}$ with $\omega_0 \in p(\mathcal{K}) (\subset \mathcal{K}_0)$.
		Then there exists a real number $a \in \mathbb{R}$ such that $\omega_0 + ae \in \mathcal{K}$. 
		We set $\omega = \omega_0 + ae$.
		By (\ref{f-2-2}), we have $a>0$.\par
		We claim that $ \omega_0 + be \in \mathcal{K} $ for any $b \in (0,a)$.
		Let $\delta \in \Delta(M)$.
		If $(\mathcal{K}, \delta)>0$, then $(\bar{\mathcal{K}}, \delta) \geqq 0$.
		Since $\omega_0 \in \bar{\mathcal{K}}$ and $\omega \in \mathcal{K}$, we have $(\omega_0, \delta) \geqq 0$ and $(\omega, \delta) > 0$.
		For any $b \in (0,a)$, we have 
		$$
			\omega_0 + be = \frac{a-b}{a} \omega_0 + \frac{b}{a} \omega,
		$$
		and $(\omega_0 + be, \delta) > 0$.
		Similarly, if $(\mathcal{K}, \delta)<0$, then $(\omega_0 + be, \delta) < 0$ for any $b \in (0,a)$.
		Since $\mathcal{K}$ is a connected component of $\tilde{\mathscr{C}}_M \setminus \bigcup_{\delta \in \Delta(M)} \delta^{\perp}$,
		we have 
		$$
			\omega_0 + be \in \mathcal{K}
		$$
		for any $b \in (0,a)$.\par
		\noindent
		\textbf{Step 2 } 
		Let $g \in O^+(M_0^{\perp})$. 
		By Lemma \ref{l-2-7}, there exists an element $\tau \in \Gamma(\mathcal{K}_0)$ such that $\tau|_{M_0^{\perp}} =g$.
		We define $\sigma \in O(L_2)$ by
		$$
			\sigma( x_0 + ae) = \tau(x_0) + ae, \quad (x_0 \in L_{K3}, a \in \mathbb{Z}).
		$$
		By the construction of $\sigma$ and $\iota_M$, we have $\sigma \circ \iota_M = \iota_M \circ \sigma$.
		If we write $\tau =s_{v_1} \dots s_{v_m}$ $(v_i \in L_{K3, \mathbb{R}})$, 
		then $v_i \in L_2$ and $\sigma =s_{v_1} \dots s_{v_m}$.
		Therefore $sn_{\mathbb{R}}( \sigma ) =sn_{\mathbb{R}}( \tau ) =+1$
		and we have $\sigma \in O^+(L_2) = \mon(L_2)$.
		Therefore we have $\sigma \in \Gamma(M)$.\par
		Suppose that $\mathcal{K} \cap \sigma(\mathcal{K}) = \emptyset$.
		Then there exists $\delta \in \Delta(M)$ such that 
		\begin{align}\label{f-2-3}
			(\mathcal{K}, \delta)>0 \quad \text{and} \quad (\sigma(\mathcal{K}), \delta)<0.
		\end{align}
		Write $\delta = d + \lambda e$, where $d \in M_{0, \mathbb{R}}$ and $\lambda \in \mathbb{Z}$.
		We define a real-valued function $f(t)$ on $[0,1]$ by
		$$
			f(t) =(t \tau(\omega_0) + (1-t) \omega_0, d), \quad ( t\in [0,1] ).
		$$
		Since $\tau \in \Gamma(\mathcal{K}_0)$, we have $\tau(\omega_0) \in \mathcal{K}_0$.
		Since $\mathcal{K}_0$ is a convex set, the line segment $\left\{ t \tau(\omega_0) + (1-t) \omega_0 ; 0 \leqq t \leqq 1 \right\}$ is contained in $\mathcal{K}_0$.
		By Lemma \ref{l-2-8}, the hyperplane $d^{\perp}$ in $M_{0, \mathbb{R}}$ does not intersect with $\tilde{\mathscr{C}}_{M_0} \setminus  \bigcup_{\delta \in \Delta(M_0)} \delta^{\perp}$.
		Therefore, we have $(x_0, d) \neq 0$ for any $x_0 \in \mathcal{K}_0$.
		Hence $f(t) \neq 0$ for any $t \in [0,1]$.
		Therefore, there exists a positive constant $C>0$ such that $|f(t)| \geqq C$ for any $0 < t<1$.
		 Let $a' > 0$ be a positive number that satisfies
		 $$
		 	 a' < \left\{ 
					\begin{split}
						&\operatorname{min} \left\{ a, \frac{C}{2|\lambda|} \right\} &\quad  \text{if} \quad \lambda \neq 0 \\
						&\qquad a & \text{if} \quad \lambda = 0
				\end{split} \right.
		 $$
		 and set $\omega' = \omega_0 + a' e$.
		 By Step 1, $\omega_0 + b e \in \mathcal{K}$ for any $b \in (0, a)$.
		 In particular, $\omega' \in \mathcal{K}$.
		Hence $\sigma(\omega') \in \sigma(\mathcal{K})$.
		 By (\ref{f-2-3}), there is a real number $0<t_0<1$ such that
		 $$
		 	0=(t_0 \sigma(\omega') +(1-t_0)\omega', \delta)=f(t_0)-2a' \lambda.
		 $$
		 If $\lambda = 0$, we have $0=|f(t_0)| \geqq C >0$, which is impossible.
		 Suppose that $\lambda \neq 0$.
		 Since $|f(t_0)| \geqq C$ and $a' < \frac{C}{2 |\lambda|}$, we get 
		 $$
		 	C \leqq |f(t_0)| = 2a' |\lambda|<C,
		 $$
		 which is also impossible.
		 Thus we have $\mathcal{K} \cap \sigma(\mathcal{K}) \neq \emptyset$.
		 Since $\mathcal{K}$ and $\sigma(\mathcal{K})$ are connected components of $\tilde{\mathscr{C}}_M \setminus \bigcup_{\delta \in \Delta(M)} \delta^{\perp}$,
		 we have $\mathcal{K}=\sigma(\mathcal{K})$, and $\sigma \in \Gamma(\mathcal{K})$.
		 Therefore $g = \tau|_{M_0^{\perp}}= \sigma|_{M^{\perp}} \in \Gamma_{M^{\perp}, \mathcal{K}}$.
		 Since $g \in O^+(M_0^{\perp})$ is arbitrary, we have $\Gamma_{M^{\perp}, \mathcal{K}} = O^+(M_0^{\perp})$.
		 If $(\mathcal{K}, e)>0$, then we can prove the statement in the same manner.
	\end{proof}
		
	\begin{cor}\label{c-2-1}
		Let $M_0$ be a primitive hyperbolic 2-elementary sublattice of $L_{K3}$ and set $M=M_0 \oplus \mathbb{Z} e$.
		Let $\mathcal{K} \in \ktm$ be a natural chamber.
		The identity map $\Omega_{M_0^{\perp}} \to \Omega_{M^{\perp}}$ induces an isomorphism of orthogonal modular varieties $\mathcal{M}_{M_0} \cong \mathcal{M}_{M, \mathcal{K}}$.
	\end{cor}
	
	Let $\tilde{\mathcal{M}}_{M_0}$ be the set of isomorphism classes of 2-elementary K3 surfaces of type $M_0$,
	and $\bar{\pi}_{M_0} : \tilde{\mathcal{M}}_{M_0} \to \mathcal{M}_{M_0}$ be the period map defined by
	$$
		\bar{\pi}_{M_0}(Y, \sigma) =  [\alpha(H^{2,0}(Y))] \quad ((Y, \sigma) \in \tilde{\mathcal{M}}_{M_0}),
	$$
	where $\alpha : H^2(Y, \mathbb{Z}) \to L_{K3}$ is a marking of $(Y, \sigma)$.
	Recall that $\tilde{\mathcal{M}}_{M, \mathcal{K}}$ is the set of isomorphism classes of $K3^{[2]}$-type manifolds with involution of type $(M, \mathcal{K})$,
	and $P_{M, \mathcal{K}} : \tilde{\mathcal{M}}_{M, \mathcal{K}} \to \mathcal{M}_{M, \mathcal{K}}$ is the period map.
	We have a natural map $\Phi : \tilde{\mathcal{M}}_{M_0} \to \tilde{\mathcal{M}}_{M, \mathcal{K}}$ defined by
	$$
		\Phi(Y, \sigma) = (Y^{[2]}, \sigma^{[2]})  \quad ((Y, \sigma) \in \tilde{\mathcal{M}}_{M_0}).
	$$
	Then the following diagram commutes:
	$$
		\xymatrix{
			\tilde{\mathcal{M}}_{M_0} \ar[r]^{\Phi}  \ar[d]_{\bar{\pi}_{M_0}} & \tilde{\mathcal{M}}_{M, \mathcal{K}} \ar[d]^{P_{M, \mathcal{K}}}   \\
			 \mathcal{M}_{M_0} \ar[r]_{\cong} & \mathcal{M}_{M, \mathcal{K}}  
		}
	$$
	
	\begin{exa}\label{e-2-35}
		Let $C$ be a smooth sextic in $\mathbb{P}^2$ and let $Y \to \mathbb{P}^2$ be the double cover blanched over $C$.
		The covering involution of $Y \to \mathbb{P}^2$ is denoted by $\sigma : Y \to Y$.
		By Example \ref{e-2-4}, $(Y^{[2]}, \sigma^{[2]})$ and $(\elm_{Y/\sigma}(Y^{[2]}), \elm_{Y/\sigma}(\sigma^{[2]}))$ are manifolds of $K3^{[2]}$-type with antisymplectic involution.
		The invariant lattice $H^2(Y, \mathbb{Z})^{\sigma}$ is generated by an element $h_0$ with $h_0^2=+2$.
		By Example \ref{e-2-13}, we have
		$
			H^2(Y^{[2]}, \mathbb{Z})^{\sigma^{[2]}} \cong \mathbb{Z}h_0 \oplus \mathbb{Z} \varepsilon.
		$
		By \cite[Proposition 25.14]{MR1963559}, we also have $H^2(\elm_{Y/\sigma}(Y^{[2]}), \mathbb{Z})^{\elm_{Y/\sigma}(\sigma^{[2]})} \cong \mathbb{Z}h_0 \oplus \mathbb{Z} \varepsilon.$\par
		We define a sublattice $M_0$ of $L_{K3}$ by $M_0 =\mathbb{Z}h$, where $h \in L_{K3}$ satisfies $h^2=2$.
		It is a primitive hyperbolic 2-elementary sublattice of $L_{K3}$.
		We set $M=M_0 \oplus \mathbb{Z} e$.
		Then $M$ is an admissible sublattice of $L_2$.
		By \cite[Example 9.12]{MR3519981}, we have
		$$
			\Delta(M) = \pm \{e, 2h+3e, 2h-3e \} \quad \text{and} \quad \ktm/\Gamma_{M} =\{ [\mathcal{K}], [\mathcal{K}'] \},
		$$
		where $\mathcal{K}$ and $\mathcal{K}'$ are K\"ahler-type chambers defined by $\mathcal{K}= \mathbb{R}_{> 0} h +\mathbb{R}_{> 0}(3h+2e)$ and $\mathcal{K}' =\mathbb{R}_{> 0}(3h+2e) +\mathbb{R}_{> 0}(h+e)$.
		By \cite[Example 9.12]{MR3519981}, $\mathcal{K}$ is natural, $\rho(Y^{[2]}, \sigma^{[2]})=[\mathcal{K}]$ and $\rho(\elm_{Y/\sigma}(Y^{[2]}), \elm_{Y/\sigma}(\sigma^{[2]}))=[\mathcal{K}']$.\par
		By Proposition \ref{p-2-1} and Corollary \ref{c-2-1}, we have $\Gamma_{M^{\perp}, \mathcal{K}} = O^+(M_0^{\perp})$
		and the identity map $\Omega_{M_0^{\perp}} \to \Omega_{M^{\perp}}$ gives an isomorphism of $\mathcal{M}_{M_0} \cong \mathcal{M}_{M, \mathcal{K}}$.
	\end{exa}

	We prove that the same statement holds for the non-natural K\"ahler-type chamber $[\mathcal{K}']$.
	
	\begin{prop}\label{p-2-2}
		With the same notation as in Example \ref{e-2-35}, we have $\Gamma_{M^{\perp}, \mathcal{K}'} = O^+(M_0^{\perp})$.
		In particular, the identity map $\Omega_{M_0^{\perp}} \to \Omega_{M^{\perp}}$ induces an isomorphism of orthogonal modular varieties $\mathcal{M}_{M_0} \cong \mathcal{M}_{M, \mathcal{K}'}$.
	\end{prop}
	
	\begin{proof}
		It suffices to show that $\Gamma_{M^{\perp}, \mathcal{K}'} = \Gamma_{M^{\perp}, \mathcal{K}}$.
		Let $g \in \Gamma_{M^{\perp}, \mathcal{K}'}$.
		There exists an element $\sigma \in \Gamma(\mathcal{K}')$ such that $\sigma|_{M^{\perp}} =g$.
		Since the boundary $\partial \mathcal{K}'$ consists of the rays $\mathbb{R}_{\geqq 0}(3h+2e)$ and $\mathbb{R}_{\geqq 0}(h+e)$ and since $\sigma|_{M}$ preserves $\partial \mathcal{K}'$, we have
		$$
			(\sigma|_{M})(3h+2e)=3h+2e \text{ and } (\sigma|_{M})(h+e)=h+e.
		$$ 
		Hence $\sigma|_{M} = \operatorname{id}_M$.
		Therefore $\sigma \in \Gamma(\mathcal{K})$ and $g = \sigma|_{M^{\perp}} \in \Gamma_{M^{\perp}, \mathcal{K}}$.
		Similarly we have $\Gamma_{M^{\perp}, \mathcal{K}'} \supset \Gamma_{M^{\perp}, \mathcal{K}}$, which completes the proof.
	\end{proof}
	
	Note that a birational transformation of an irreducible holomorphic symplectic manifold preserves its period (\cite[Proposition 25.14]{MR1963559}).
	We have a natural map $\Psi : \tilde{\mathcal{M}}_{M_0} \to \tilde{\mathcal{M}}_{M, \mathcal{K}'}$ defined by
	$$
		\Psi(Y, \sigma) = (\elm_{Y/\sigma}(Y^{[2]}), \elm_{Y/\sigma}(\sigma^{[2]}))  \quad ((Y, \sigma) \in \tilde{\mathcal{M}}_{M_0}).
	$$
	Then the following diagram commutes:
	$$
		\xymatrix{
			\tilde{\mathcal{M}}_{M_0} \ar[r]^{\Psi}  \ar[d]_{\bar{\pi}_{M_0}} & \tilde{\mathcal{M}}_{M, \mathcal{K}'} \ar[d]^{P_{M, \mathcal{K}'}}   \\
			 \mathcal{M}_{M_0} \ar[r]_{\cong} & \mathcal{M}_{M, \mathcal{K}'} . 
		}
	$$

\section{An invariant of $K3^{[2]}$-type manifolds with antisymplectic involution}\label{s-3}

\subsection{Some fundamental properties of hyperk\"ahler manifolds}\label{ss-3-1}

	Throughout this section, we fix an admissible sublattice $M$ of $L_2$ and a K\"ahler-type chamber $\mathcal{K} \in \ktm$.
	
	Let $(X, \iota)$ be a manifold of $K3^{[2]}$-type with involution of type $(M, \mathcal{K})$.
	Choose $\eta \in H^0(X, \Omega_X^2)$ and $\theta \in H^2(X, \mathcal{O}_X)$.
	
	Let $h_{X,0}$ be an $\iota$-invariant Ricci-flat \K metric on $X$
	with \K form $\omega_{X,0}$.
	It is locally written as
	$$
		\omega_{X,0} = \frac{i}{2} \sum_{j,k} h_{X,0} \left( \frac{\partial}{\partial z^j}, \frac{\partial}{\partial z^k} \right) dz^j \wedge d\bar{z}^k .
	$$
	The Riemannian metric associated with $h_{X,0}$ is denoted by $g$.
	The hermitian metric on $\wedge^{p,q} T^*X$ attached to the Ricci-flat \K metric is also donoted by $h_{X,0}$
	
	Let $I$ be the complex structure of $X$.
	Since $(X,g)$ is hyperk\"ahler, there are complex structures $J$ and $K$ of $X$ such that $(X, I, g)$, $(X, J, g)$ and $(X, K, g)$ are manifolds of $K3^{[2]}$-type and 
	$
		IJ=-JI=K.
	$
	The \K forms with respect to $J$ and $K$ are given by
	$$
		\omega_J =g(-, J(-)), \quad \text{and} \quad \omega_K =g(-, K(-)),
	$$
	respectively.
	Set $\sigma_I = \omega_J +i \omega_K$.
	This is a holomorphic 2-form on $X$ (cf. \cite[\S 23]{MR1963559}).
	Since $H^0(X, \Omega_X^2) =\mathbb{C} \eta$, there exists a complex number $\lambda \in \mathbb{C}$ such that $\eta = \frac{\lambda}{2} \sigma_I$.
	Note that the $L^2$-norm of $\eta^2$ is given by
	$$
		||\eta^2||_{L^2}^2 =  \int_X h_{X,0}(\eta^2, \eta^2) \frac{\omega_{X,0}^4}{4!} = \int_X \eta^2 \wedge \bar{\eta}^2 ,
	$$
	and the volume of $(X, \omega_{X,0})$ is defined by
	$$
		\vol(X, \omega_{X,0}) = \int_X \frac{\omega_{X,0}^4}{4!}.
	$$
	Since $\omega_{X,0}$ is Ricci-flat, it follows from \cite[Corollary 23.9]{MR1963559} that
	\begin{align}\label{f-3-1}
		|\lambda|^2 = \frac{1}{2} \left( \frac{\eta^2 \wedge \bar{\eta}^2}{\omega_{X,0}^4/4!} \right)^{\frac{1}{2}} =\frac{1}{2} \left( \frac{||\eta^2||_{L^2}^2}{\vol(X, \omega_{X,0})} \right)^{\frac{1}{2}}.
	\end{align}
	Similarly, there exists a complex number $\mu \in \mathbb{C}$ such that $\theta = \frac{\mu}{2} \bar{\sigma}_I$ in $H^2(X, \mathcal{O}_X)$.
	We identify the cohomology class $\theta$ with its harmonic representative.
	 The $L^2$-norm of $\theta^2$ is given by
	$$
		||\theta^2||_{L^2}^2 = \int_X h_{X,0}(\theta^2, \theta^2) \frac{\omega_{X,0}^4}{4!} = \int_X \theta^2 \wedge \bar{\theta}^2 ,
	$$
	and we have
	\begin{align}\label{f-3-2}
		|\mu|^2  =\frac{1}{2} \left( \frac{||\theta^2||_{L^2}^2}{\vol(X, \omega_{X,0})} \right)^{\frac{1}{2}}.
	\end{align}
	
	\begin{lem}\label{l-3-1}
		For any $\alpha \in A^{1,1}(X)$, the following identity holds:
		$$
			h_{X,0}(\theta \wedge \alpha, \theta \wedge \alpha) = |\mu|^2 h _{X,0}(\alpha, \alpha).
		$$
	\end{lem}
	
	\begin{proof} 
		Fix $p \in X$.
		Since $(X, g)$ is hyperk\"ahler, the real tangent space $T_{p, \mathbb{R}}X$ at $p$ is equipped with the structure of a quaternionic hermitian vector space
		(see \cite[\S 23.2]{MR1963559}).
		Therefore, there exist two tangent vector $e, f \in T_{p, \mathbb{R}}X$ such that 
		$$
			e, Ie, Je, Ke, f, If, Jf, Kf
		$$
		form an orthonormal basis of $(T_{p, \mathbb{R}}X, g)$.
		Set 
		$$
			v_1 =\frac{ e-iIe }{2}, \quad v_2 =\frac{ -Ke-iJe }{2}, \quad v_3 =\frac{ f-iIf }{2}, \quad v_4 =\frac{ -Kf-iJf }{2}.
		$$
		Then $v_1, v_2, v_3, v_4$ are of type $(1,0)$ with respect to the complex structure $I$
		and form a $\mathbb{C}$-basis of the holomorphic tangent space $T_pX$ at $p$.
		Moreover, we have
		$$
			h_{X,0}(v_i, v_j) = \frac{1}{2} \delta_{ij}.
		$$
		Let $v^1, v^2, v^3, v^4$ be the dual basis. 
		Then we have
		$$
			h_{X,0}(v^i, v^j) = 2\delta_{ij}, \quad \text{and} \quad \sigma_I=i v^1 \wedge v^2 +i v^3 \wedge v^4.
		$$
		
		Let $\alpha = \frac{1}{2} \sum_{i,j} \alpha_{ij} v^i \wedge \bar{v}^j \in \wedge^{1,1}T_p^*X$.
		Then we have
		$$
			h_{X,0}(\bar{\sigma}_I \wedge \alpha, \bar{\sigma}_I \wedge \alpha) = 4\sum_{i,j} |\alpha_{ij}|^2 = 4 h _{X,0}(\alpha, \alpha),
		$$
		which completes the proof.
	\end{proof}

\subsection{Construction of an invariant}\label{ss-3-2}

	We now turn to the case where the $\iota$-invariant \K metric $h_X$ is not necessarily Ricci-flat. 
	
	The \K form attached to $h_X$ is locally defined by
	$$
		\omega_X = \frac{i}{2} \sum_{j,k} h_X \left( \frac{\partial}{\partial z^j}, \frac{\partial}{\partial z^k} \right) dz^j \wedge d\bar{z}^k ,
	$$
	where $z^1, \dots, z^4$ is a local coordinate on $X$.
	The volume of $(X, \omega_{X})$ is defined by
	$$
		\vol(X, \omega_{X}) = \int_X \frac{\omega_{X}^4}{4!}.
	$$
	Set $\omega_{X^{\iota}} = \omega_X|_{X^{\iota}}$.
	This is a \K form on $X^{\iota}$ attached to $h_{X^{\iota}} =h_X|_{X^{\iota}}$.
	Recall that $X^{\iota}$ is a possibly disconnected compact complex surface.
	Let $X^{\iota} = \sqcup_i Z_i$ be the decomposition into the connected components.
	We define the volume of $(X^{\iota}, \omega_{X^{\iota}})$ by
	$$
		\vol(X^{\iota}, \omega_{X^{\iota}}) =\prod_i \vol(Z_i, \omega_X|_{Z_i}) = \prod_i \int_{Z_i} \frac{(\omega_X|_{Z_i})^2}{2!}.
	$$
	The covolume of the lattice $\operatorname{Im}(H^1(X^{\iota}, \mathbb{Z}) \to H^1(X^{\iota}, \mathbb{R}))$ with respect to the $L^2$-metric induced from $h_X$ is denoted by $\vol_{L^2}(H^1(X^{\iota}, \mathbb{Z}), \omega_{X^{\iota}})$ .
	Namely, 
	$$
		\vol_{L^2}(H^1(X^{\iota}, \mathbb{Z}), \omega_{X^{\iota}}) = \det(\langle e_i, e_j \rangle_{L^2}),
	$$
	where $e_1, \dots ,e_{b_1(X^{\iota})}$ is an integral basis of $\operatorname{Im}(H^1(X^{\iota}, \mathbb{Z}) \to H^1(X^{\iota}, \mathbb{R}))$.
	
	We define a real-valued function $\varphi$ on $X^{\iota}$ by
	$$
		\varphi = \frac{||\eta^2||_{L^2}^2}{\eta^2 \wedge \bar{\eta}^2} \frac{\omega_{X}^4/4!}{\vol(X, \omega_{X})}.
	$$
	Obviously, $\varphi$ is independent of the choice of $\eta$.
	We define a positive number $A(X, \iota, h_X) \in \mathbb{R}_{>0}$ by
	$$
		A(X, \iota, h_X) =\exp \left[ \frac{1}{48} \int_{X^{\iota}} (\log \varphi) \Omega \right],
	$$
	where $\Omega$ is a characteristic form on $X^{\iota}$ defined by
	$$
		\Omega = c_1(TX^{\iota}, h_{X^{\iota}})^2 -8c_2(TX^{\iota}, h_{X^{\iota}}) -c_1(TX, h_X)|^2_{X^{\iota}} +3c_2(TX, h_X)|_{X^{\iota}}.
	$$
	Here we denote by $c_i(TX, h_X)$, $c_i(TX^{\iota}, h_{X^{\iota}})$ the $i$-th Chern form of the hermitian holomorphic vector bundles $(TX, h_X)$, $(TX^{\iota}, h_{X^{\iota}})$, respectively.
	Note that if $h_X$ is Ricci-flat, then we have $\varphi =1$ and $A(X, \iota, h_X) =1$.
	
	Recall that $t=\operatorname{Tr}(\iota^*|_{H^{1,1}(X)})$.
	By the definition of the admissible sublattice $(M, \iota_M)$, we have $t= \operatorname{Tr}(\iota_M)+2$.
	Therefore $t$ depends only on $(M, \iota_M)$, and is independent of $(X, \iota)$ itself.
	
	Let $\tau_{\iota}(\bar{\Omega}_X^1)$ be the equivariant analytic torsion of the cotangent bundle $\bar{\Omega}_X^1 =(\Omega_X^1, h_X)$ endowed with the hermitian metric induced from $h_X$,
	and let $\tau(\bar{\mathcal{O}}_{X^{\iota}})$ be the analytic torsion of the trivial bundle $\bar{\mathcal{O}}_{X^{\iota}}$ with respect to the canonical metric.
	
	\begin{dfn}\label{d-3-1}
		We define a real number $\tau_{M, \mathcal{K}}(X, \iota)$ by
		\begin{align*}
			\tau_{M, \mathcal{K}}(X, \iota)=\tau_{\iota}(\bar{\Omega}_X^1) &\vol(X, \omega_{X})^{\frac{(t-1)(t-7)}{16}} A(X, \iota, h_X) \\
			&\cdot \tau(\bar{\mathcal{O}}_{X^{\iota}})^{-2} \vol(X^{\iota}, \omega_{X^{\iota}})^{-2} \vol_{L^2}(H^1(X^{\iota}, \mathbb{Z}), \omega_{X^{\iota}}).
		\end{align*}
	\end{dfn}

\subsection{Properties of $\tau_{M, \mathcal{K}}$}
	
	Let $\X$ be a complex manifold with holomorphic involution $\iota : \X \to \X$,
	let $S$ be a complex manifold,
	and let $f : (\X, \iota) \to S$ be a family of $K3^{[2]}$-type manifolds with involution of type $(M, \mathcal{K})$.
	
	The fixed locus of $\iota : \X \to \X$ is denoted by $\XX = \{ x \in \X ; \iota(x) =x \}$.
	The restriction of $f : \X \to S$ to $\XX$ is also denoted by $f : \XX \to S$
	and it is a family of smooth complex surfaces.
	
	Let $h_{\xs}$ be an $\iota$-invariant fiberwise K\"ahler metric on $T\xs$.
	For $s \in S$, the K\"ahler form associated with the K\"ahler metric $h_s = h_{\xs}|_{X_s}$ is denoted by $\omega_s$.
	We set $\omega_{\xs} = \{ \omega_s \}_{s \in S}$.
	Let $h_{\xss}$ and $h_N$ be the induced and quotient metric on $T\xss$ and $N_{\XX/\X}$, respectively.
	The $\iota$-invariant hermitian metric on $\Omega^1_{\xs}$ induced from $h_{\xs}$ is also denoted by $h_{\xs}$.
	
	Let $E(\pm 1)$ be the $(\pm 1)$-eigenbundle of the $\mu_2$-action of $T{\xs}|_{\mathscr{X}^{\iota}}$.
	They are holomorphic vector bundles on $\mathscr{X}^{\iota}$, 
	and the decomposition $T{\xs}|_{\mathscr{X}^{\iota}} = E(+1) \oplus E(-1)$ is orthogonal with respect to the metric $h_{\xs}$.
	The restriction of $h_{\xs}$ to $E(\pm 1)$ is denoted by $h_{\pm}$.
	Then we have two isometries 
	\begin{align}\label{f-3-A}
		(T{\xss}, h_{\xss}) \cong (E(+1), h_+) \quad \text{and} \quad (N_{\mathscr{X}^{\iota}/ \mathscr{X}}, h_N) \cong (E(-1), h_-).
	\end{align}
	Hence the following short exact sequence of holomorphic hermitian vector bundles splits
	\begin{align}\label{f-3-3}
		0 \to (T{\xss}, h_{\xss}) \to (T{\xs}, h_{\xs})|_{\mathscr{X}^{\iota}} \to (N_{\mathscr{X}^{\iota}/ \mathscr{X}}, h_N) \to 0.
	\end{align}
	
	We set
	$$
		\overline{T}{\xss} = (T{\xss}, h_{\xss}), \overline{T}{\xs}=(T{\xs}, h_{\xs}) , \text{and } \overline{N}_{\mathscr{X}^{\iota}/ \mathscr{X}}=(N_{\mathscr{X}^{\iota}/ \mathscr{X}}, h_N).
	$$
	Let
	$$
		c_i(\overline{T}{\xss}), \quad c_i(\overline{T}{\xs}), \quad \text{and} \quad c_i(\overline{N}_{\mathscr{X}^{\iota}/ \mathscr{X}})
	$$
	be their Chern forms, respectively.
	
	\begin{lem}\label{l-3-2}
		The following identities hold:
		$$
			\az = -\ax+\au,
		$$
		$$
			\aw = \ax^2 -\ay -\ax\au +\av.
		$$
	\end{lem}
	
	\begin{proof}
		Since the short exact sequence (\ref{f-3-3}) splits, we have
		$$
			c(\overline{T}{\xs})|_{\mathscr{X}^{\iota}} = c(\overline{T}{\xss} )c(\overline{N}_{\mathscr{X}^{\iota}/ \mathscr{X}}),
		$$ 
		and this implies the statement.
	\end{proof}
	
	We define a characteristic form $\Omega \in A^{2,2}(\XX)$ by
	$$
		\Omega= \ax^2 -8\ay -\au^2 +3\av.
	$$
	
	\begin{lem}\label{l-3-3}
		The following identities holds in $A^{3,3}(\XX)$:
		$$
			\left[ Td_{\iota}(\overline{T}{\xs}) ch_{\iota}(\overline{\Omega}^1_{\xs}) \right]^{(3,3)} =2\left[ Td(\overline{T}{\xss} ) \right]^{(3,3)} +\frac{1}{48} \au \wedge \Omega.
		$$
	\end{lem}
	
	\begin{proof}
		By the construction of the equivariant Todd form (\ref{al-1-A}) and by the isometries (\ref{f-3-A}), we have
		\begin{align}\label{al-3-A}
			Td_{\iota}(\overline{T}{\xs}) = Td(\overline{T}{\xss})  \det \left( \frac{I}{I+ \exp({+\frac{R_-}{2\pi i}})} \right), 
		\end{align}
		where $R_-$ is the curvature form of $(N_{\mathscr{X}^{\iota}/ \mathscr{X}}, h_N)$.
		By the isometries (\ref{f-3-A}), $\Omega^1_{\xss}$ is the $(+1)$-eigenbundle of $\Omega^1_{\xs}|_{\mathscr{X}^{\iota}}$ and $N^{\vee}_{\mathscr{X}^{\iota}/ \mathscr{X}}$ is the $(-1)$-eigenbundle of $\Omega^1_{\xs}|_{\mathscr{X}^{\iota}}$.
		By the construction of the equivariant Chern character form (\ref{al-1-B}), we have
		\begin{align}\label{al-3-D}
			ch_{\iota}(\overline{\Omega}^1_{\xs}) =ch(\overline{\Omega}^1_{\xss}) -ch(\overline{N}^{\vee}_{\mathscr{X}^{\iota}/ \mathscr{X}}). 
		\end{align} \par
		By the definition of Todd form and Chern form, we have
		\begin{align}\label{al-3-B}
		\begin{aligned}
			Td(\overline{T}{\xss}) = 1 &+\frac{1}{2}\ax +\frac{1}{12}\{ \ax^2+\ay \} \\
			& +\frac{1}{24} \ax\ay +\text{higher degree terms}.
		\end{aligned}
		\end{align}
		Since $\frac{1}{1+e^{-x}} = \frac{1}{2} +\frac{x}{4} +0x^2 +\dots$, the following identity of functions of $2 \times 2$ matrices holds:
		$$
			\det \left( \frac{I}{I+ \exp(-A)} \right) = \frac{1}{4} + \frac{1}{8} c_1(A) 
			 +\frac{1}{16} c_2(A) +\text{higher degree terms} .
		$$
		Therefore, we have
		\begin{align}\label{al-3-C}
		\begin{aligned}
			\det \left( \frac{I}{I+ \exp({+\frac{R_-}{2\pi i}})} \right)
			=\frac{1}{4} +\frac{1}{8}\az +\frac{1}{16}\aw +\text{higher degree terms}.	
		\end{aligned}
		\end{align}
		By (\ref{al-3-A}), (\ref{al-3-B}), (\ref{al-3-C}), we have
		\begin{align}\label{f-3-4}
		\begin{aligned}
			&[Td_{\iota}(\overline{T}{\xs})]^{(0,0)} =\frac{1}{4}, \\
			&[Td_{\iota}(\overline{T}{\xs})]^{(1,1)} =\frac{1}{8} \{ \ax +\az \}, \\
			&[Td_{\iota}(\overline{T}{\xs})]^{(2,2)} =\frac{1}{48} \{ \ax^2 +\ay +3\ax\az +3\aw \}.
		\end{aligned}
		\end{align}
		On the other hand, we have
		\begin{align}\label{al-3-E}
		\begin{aligned}
			ch(\overline{\Omega}^1_{\xss})=& 2 -\ax +\frac{1}{2}\{ \ax^2 -2\ay \} \\
			&-\frac{1}{6}\{ \ax^3 -3\ax\ay \} +\text{higher degree terms},
		\end{aligned}
		\end{align}
		and
		\begin{align}\label{al-3-F}
		\begin{aligned}
			ch(\overline{N}^{\vee}_{\mathscr{X}^{\iota}/ \mathscr{X}}) =& 2 -\az +\frac{1}{2}\{ \az^2-2\aw \} \\
			&-\frac{1}{6}\{ \az^3 -3\az\aw \} +\text{higher degree terms}
		\end{aligned}
		\end{align}
		By (\ref{al-3-D}), (\ref{al-3-E}), (\ref{al-3-F}), we  have
		\begin{align}\label{f-3-5}
		\begin{aligned}
			&[ch_{\iota}(\overline{\Omega}^1_{\xs})]^{(0,0)} = 0 ,\\
			&[ch_{\iota}(\overline{\Omega}^1_{\xs})]^{(1,1)} = -\ax +\az ,\\
			&[ch_{\iota}(\overline{\Omega}^1_{\xs})]^{(2,2)} = \frac{1}{2}\{ \ax^2 -2\ay -\az^2 +2\aw \} ,\\
			&[ch_{\iota}(\overline{\Omega}^1_{\xs})]^{(3,3)} = -\frac{1}{6}\{ \ax^3 -3\ax\ay \\
			&\hspace{120pt} -\az^3 +3\az\aw \}.
		\end{aligned}
		\end{align}
		Combining the formulas (\ref{f-3-4}) and (\ref{f-3-5}), we obtain
		\begin{align*}
			&\left[ Td_{\iota}(\overline{T}{\xs}) ch_{\iota}(\overline{\Omega}^1_{\xs}) \right]^{(3,3)} \\
			&=\frac{1}{48} \left\{  \ax^2\az -\az^3 -\ax\ay \right. \\
			&\qquad \qquad \left. +3\ax\aw +3\az\aw -5\az\ay \right\}.
		\end{align*}
		By Lemma \ref{l-3-2} and by (\ref{al-3-B}), we obtain the desired formula.
	\end{proof}
	
	Consider the direct image sheaf $f_* K_{\xs}$ of the relative canonical bundle $K_{\xs}$.
	This is a holomorphic line bundle on $S$ and is equipped with the $L^2$-metric $h_{L^2}$.
	On the other hand, $K_{\xs}$ is equipped with the hermitian metric $h$ induced from the fiberwise \K metric $h_{\xs}$.
	We define a smooth function $\varphi$ on $S$ by
	$$
		\varphi = \frac{ \omega^4_{\xs} /4! }{ \eta^2 \wedge \bar{\eta}^2 } \frac{ || \eta^2 ||_{L^2}^2 }{ \vol \left( \xs, \omega_{\xs} \right) },
	$$
	where $\eta$ is a nowhere-vanishing holomorphic section of $f_* \Omega^2_{\xs}$.
	Then $\varphi$ is independent of the choice of $\eta$.
	The evaluation map $f^*f_*K_{\xs} \to K_{\xs}$ is an isomorphism and we have
	\begin{align}\label{f-3-6}
		c_1(K_{\xs}, h) =f^* c_1( f_*K_{\xs}, h_{L^2}) +dd^c \log \varphi + dd^c \log \vol \left( \xs, \omega_{\xs} \right)
	\end{align}
	in $A^{1,1}(\XX)$.
	We define a function $A(\xs)$ on $S$ by
	$$
		A(\xs)(s) = A(X_s, \iota_s, h_s) \qquad (s \in S).
	$$
	
	\begin{lem}\label{l-3-1-A}
		The following identity holds:
		$$
			\int_{\xss} \Omega =-3(t^2+7)
		$$
	\end{lem}
	
	\begin{proof}
		By Lemma \ref{l-2-12}, we have
		\begin{align}\label{al-3-2-A}
			\int_{\xss} c_1(T\xss)^2 = t^2-1, \quad \int_{\xss} c_2(T\xss) = \frac{t^2+23}{2}.
		\end{align}
		Since the canonical bundle of a irreducible holomorphic symplectic manifold is trivial, we have
		\begin{align}\label{al-3-3-A}
			\int_{\xss} c_1(T\xs)|_{\XX}^2 =0.
		\end{align}
		Since a holomorphic symplectic form induces an isomorphism $N_{\XX/\X} \cong \Omega^1_{\xss}$, we have $c_1(N_{\XX/\X}) = -c_1(T\xss)$.
		By Lemma \ref{l-3-2}, we have $c_2(T\xs)|_{\XX} = 2c_2(T\xss) -c_1(T\xss)^2$.
		Hence
		\begin{align}\label{al-3-4-A}
			\int_{\xss} c_2(T\xs)|_{\XX} = 24.
		\end{align}
		By (\ref{al-3-2-A}), (\ref{al-3-3-A}), (\ref{al-3-4-A}), we have the desired result.
	\end{proof}
	
	Let $f_*$ be the integration along the fibers of $f : \XX \to S$.
	
	\begin{prop}\label{p-3-1}
		The following identity holds:
		\begin{align*}
			&\left[ f_* \left( Td_{\iota}(\overline{T}{\xs}) ch_{\iota}(\overline{\Omega}^1_{\xs}) \right) \right]^{(1,1)} \\
			&=2\left[ f_* Td(\overline{T}{\xss}) \right]^{(1,1)} -dd^c \log A(\xs) \\
			&\quad +\frac{t^2+7}{16} c_1(f_*K_{\xs}, h_{L^2}) +\frac{t^2+7}{16} dd^c \log \vol \left( \xs, \omega_{\xs} \right) .
		\end{align*}
	\end{prop}
	
	\begin{proof}
		By Lemma \ref{l-3-3} and by the formula (\ref{f-3-6}), we have
		\begin{align*}
			\left[ Td_{\iota}(\overline{T}{\xs}) ch_{\iota}(\overline{\Omega}^1_{\xs}) \right]^{(3,3)} =&2\left[ Td(\overline{T}{\xss} ) \right]^{(3,3)} -\frac{1}{48} f^* c_1(f_*K_{\xs}, h_{L^2})|_{\mathscr{X}^{\iota}}  \wedge \Omega\\
			&-\frac{1}{48} dd^c (\log \varphi)\Omega -\frac{1}{48} dd^c \left( \log \vol (\xs, \omega_{\xs}) \right) \Omega.
		\end{align*}
		By the projection formula, we have
		\begin{align*}
			\left[ f_* \left( Td_{\iota}(\overline{T}{\xs}) ch_{\iota}(\overline{\Omega}^1_{\xs}) \right) \right]^{(1,1)}
			=2\left[ f_* Td(\overline{T}{\xss}) \right]^{(1,1)} -\frac{1}{48} \int_{\xss} dd^c (\log \varphi)\Omega \\
			-\frac{1}{48} \left( \int_{\xss} \Omega \right) c_1(f_*K_{\xs}, h_{L^2}) -\frac{1}{48} \left( \int_{\xss} \Omega \right) dd^c \log \vol \left( \xs, \omega_{\xs} \right) .
		\end{align*}
		By Lemma \ref{l-3-1-A}, the proof is completed.
	\end{proof}
	
	Recall that each fiber $X$ of $f : \mathscr{X} \to S$ is a manifold of $K3^{[2]}$-type.
	By \cite[Main Theorem]{MR1879810}, we have
	\begin{align}\label{al-3-5-A}
		h^{1,q}(X) = \left\{
		 \begin{aligned} 
			&0 \quad &(q=0,2,4) \\
			&21 \quad & (q=1,3)
		 \end{aligned}
		 \right.
	\end{align}
	By \cite[Proposition 24.1]{MR1963559}, the homomorphism
	$$
		H^1(X, \Omega^1_X) \otimes H^2(X, \mathcal{O}_X) \to H^3(X, \Omega^1_X), \quad \alpha \otimes \theta \mapsto \alpha \wedge \theta
	$$
	is an isomorphism.
	Therefore we have an isomorphism of holomorphic vector bundles
	\begin{align}\label{al-3-1-A}
		R^1f_*\Omega^1_{\xs} \otimes R^2f_*\mathcal{O}_{\mathscr{X}} \cong R^3f_*\Omega^1_{\xs} .
	\end{align}
	
	For $q=1,3$, let $E(\pm 1, R^qf_*\Omega^1_{\xs} )$ be the $(\pm 1)$-eigenbundle of the $\mu_2$-action of $R^qf_*\Omega^1_{\xs} $.
	By the definition of $t$, we have
	$$
		\operatorname{rank} E(\pm 1, R^1f_*\Omega^1_{\xs} ) =\operatorname{rank} E(\mp 1, R^3f_*\Omega^1_{\xs} ) =\frac{21 \pm t}{2}.
	$$
	Since $\iota$ is fiberwise antisymplectic, we deduce from (\ref{al-3-1-A}) the following isomorphism
	$$
		E(\pm 1, R^1f_*\Omega^1_{\xs} ) \otimes R^2f_*\mathcal{O}_{\mathscr{X}} \to E(\mp 1, R^3f_*\Omega^1_{\xs} ).
	$$
	We denote by $h_{\pm}$ the hermitian metrics on $E(\pm 1, R^qf_*\Omega^1_{\xs} )$ induced from the $L^2$-metric $h_{L^2}$ on $R^qf_*\Omega^1_{\xs}$ 
	and set $R^qf_*\overline{\Omega}^p_{\xs} = ( R^qf_*\Omega^p_{\xs}, h_{L^2} )$ and $\bar{E}(\pm 1, R^qf_*\Omega^1_{\xs} ) =(E(\pm 1, R^qf_*\Omega^1_{\xs} ), h_{\pm})$.
	
	\begin{lem}\label{l-3-4}
		The following identity holds:
		\begin{align*}
			&\quad c_1(\bar{E}(\pm 1, R^1f_*\Omega^1_{\xs} )) -c_1(\bar{E}(\mp 1, R^3f_*\Omega^1_{\xs} )) \\
			&=-\frac{21 \pm t}{4} c_1(R^4f_*\bar{\mathcal{O}}_{\mathscr{X}}) -\frac{21 \pm t}{4} dd^c \log \vol (\xs, \omega_{\xs}) .
		\end{align*}
	\end{lem}
	
	\begin{proof}
		We only prove the identity
		\begin{align*}
			& \quad c_1(\bar{E}(+ 1, R^1f_*\Omega^1_{\xs} )) -c_1(\bar{E}(- 1, R^3f_*\Omega^1_{\xs} )) \\
			&=-\frac{21 + t}{4} c_1(R^4f_*\bar{\mathcal{O}}_{\mathscr{X}}) -\frac{21 + t}{4} dd^c \log \vol (\xs, \omega_{\xs}).
		\end{align*}
		The other identity can be shown in the same manner.
		Set $N= \frac{21+t}{2}$.
		Let $s_1, \dots , s_N$ be a local holomorphic frame of $E(+ 1, R^1f_*\Omega^1_{\xs} )$.
		Then $s_1 \wedge \dots \wedge s_N$ is a nowhere vanishing holomorphic section of $\det E(+1, R^1f_*\Omega^1_{\xs} )$.
		Let $\theta$ be a nowhere vanishing holomorphic section of $R^2f_*\mathcal{O}_{\mathscr{X}}$.
		Then $s_1 \wedge \theta, \dots , s_N \wedge \theta$ are local holomorphic frame of $E( -1, R^3f_*\Omega^1_{\xs} )$
		and $(s_1 \wedge \theta) \wedge \dots \wedge (s_N \wedge \theta)$ is a nowhere vanishing holomorphic section of $\det E( -1, R^3f_*\Omega^1_{\xs} )$.
		
		It suffices to show that
		\begin{align}\label{f-3-7}
		\begin{aligned}
			&-\log || (s_1 \wedge \theta) \wedge \dots \wedge (s_N \wedge \theta) ||_{L^2}^2 \\
			&=-\frac{N}{2} \log || \theta^2 ||_{L^2}^2 -\log || s_1 \wedge \dots \wedge s_N ||_{L^2}^2 +\log \left( 2^N \vol (\xs, \omega_{\xs})^{\frac{N}{2}} \right).
		\end{aligned}
		\end{align}
		We may assume that $S$ consists of one point,
		and we set $X=\mathscr{X}$.\par
		Since the $L^2$-metrics on $H^q(X, \Omega_X^1)$ $(q=1,3)$ and $H^2(X, \mathcal{O}_X)$ depend only on the choice of the \K class $[\omega_X]$ and are independent of the \K form $\omega_X$ itself,
		we may assume that $\omega_X =\omega_{X,0}$ is Ricci-flat.\par
		By Lemma \ref{l-3-1}, we have 
		$$
			h_{X,0}(\theta \wedge \alpha, \theta \wedge \alpha) = |\mu|^2 h _{X,0}(\alpha, \alpha)
		$$
		for each $\alpha \in A^{1,1}(X)$.
		By integrating both sides, we have
		$$
			||\theta \wedge \alpha||_{L^2}^2 =|\mu|^2 ||\alpha||_{L^2}^2.
		$$
		Therefore we have an isometry
		$$
			(\det H^1(X, \Omega_X^1)_+ , |\mu|^{2N} \det h_{+}) \cong (\det H^3(X, \Omega_X^1)_-, \det h_{-}).
		$$
		By the formula (\ref{f-3-2}), we obtain the formula (\ref{f-3-7}).
	\end{proof}
	
	\begin{prop}\label{p-3-2}
		The following identity holds:
		$$
			\sum_{q \geqq 0} (-1)^q [ch_{\iota} (R^qf_*\bar{\Omega}^1_{\xs})]^{(1,1)} = -\frac{t}{2} c_1(f_*K_{\xs}, h_{L^2}) +\frac{t}{2} dd^c \log \vol (\xs, \omega_{\xs}).
		$$
	\end{prop}
	
	\begin{proof} 
		By Lemma \ref{l-3-4} and by (\ref{al-3-5-A}), we have
		\begin{align*}
			\sum_{q \geqq 0} (-1)^q [ch_{\iota} (R^qf_*\bar{\Omega}^1_{\xs})]^{(1,1)} =& -[ch_{\iota} (R^1f_*\bar{\Omega}^1_{\xs})]^{(1,1)} -[ch_{\iota} (R^3f_*\bar{\Omega}^1_{\xs})]^{(1,1)} \\
			=& \Bigl. -c_1(\bar{E}(+ 1, R^1f_*\Omega^1_{\xs} )) +c_1(\bar{E}(- 1, R^1f_*\Omega^1_{\xs} )) \Bigr. \\
			& \Bigl. -c_1(\bar{E}(+ 1, R^3f_*\Omega^1_{\xs} )) +c_1(\bar{E}(- 1, R^3f_*\Omega^1_{\xs} )) \Bigr. \\
			=& \frac{t}{2} c_1(R^4f_*\bar{\mathcal{O}}_{\mathscr{X}}) +\frac{t}{2} dd^c \log \vol (\xs, \omega_{\xs}).
		\end{align*}
		By the Serre duality, we have
		\begin{align*}
			c_1(R^4f_*\bar{\mathcal{O}}_{\mathscr{X}}) =-c_1(f_*K_{\xs}, h_{L^2}),
		\end{align*}
		and we obtain the desired formula.
	\end{proof}
	
	We define the characteristic form $\omega_{H^{\cdot}(\xss)} \in A^{1,1}(S)$ by
	$$
		\omega_{H^{\cdot}(\xss)} = c_1(f_*\Omega^1_{\xss}, h_{L^2}) -c_1(R^1f_*\mathcal{O}_{\mathscr{X}^{\iota}}, h_{L^2}) -2c_1(f_*K_{\xss}, h_{L^2}).
	$$
	We define a function $\vol_{L^2}(R^1 f_* \mathbb{Z}, \omega_{\xss})$ on $S$ by
	$$
		\vol_{L^2}(R^1 f_* \mathbb{Z}, \omega_{\xss})(s) = \vol_{L^2}(H^1(X_s^{\iota}, \mathbb{Z}), \omega_s|_{X_s^{\iota}}).
	$$
	
	\begin{prop}\label{p-3-3}
		The following identity holds:
		\begin{align*}
			\sum_{q \geqq 0} (-1)^q \left[ ch(R^q f_* \bar{\mathcal{O}}_{\XX}) \right]^{(1,1)}= & -\frac{1}{2} dd^c \log \left\{ \vol(\xss, \omega_{\xss})^{2} \vol_{L^2}(R^1 f_* \mathbb{Z}, \omega_{\xss})^{-1}  \right\} \\
			&+\frac{1}{2} \omega_{H^{\cdot}(\xss)}.
		\end{align*}
	\end{prop}
	
	\begin{proof}
		We have
		\begin{align}\label{f3-3-1}
			c_1(f_* \bar{\mathcal{O}}_{\mathscr{X}^{\iota}}) =-dd^c \log \vol (\xss, \omega_{\xss}), \quad \text{ and } \quad c_1(R^2 f_* \bar{\mathcal{O}}_{\mathscr{X}^{\iota}}) =-c_1(f_* K_{\xss}, h_{L^2}).
		\end{align}
		Moreover, since
		$$
			c_1(f_* \bar{\Omega}^1_{\xss}) + c_1(R^1 f_* \bar{\mathcal{O}}_{\mathscr{X}^{\iota}}) = c_1(R^1f_*\mathbb{C} \otimes \mathcal{O}_S, h_{L^2})=-dd^c \log \vol_{L^2}(R^1 f_* \mathbb{Z}, \omega_{\xss}),
		$$
		we have
		\begin{align}\label{f3-3-2}
			c_1(R^1 f_* \bar{\mathcal{O}}_{\mathscr{X}^{\iota}}) =\frac{1}{2} c_1(R^1 f_* \bar{\mathcal{O}}_{\mathscr{X}^{\iota}})  -\frac{1}{2} c_1(f_* \bar{\Omega}^1_{\xss})  -\frac{1}{2} dd^c \log \vol_{L^2}(R^1 f_* \mathbb{Z}, \omega_{\xss}).
		\end{align}
		By the formulas (\ref{f3-3-1}) and (\ref{f3-3-2}), the proof is completed. 
	\end{proof}
	
	We show that a family of $K3^{[2]}$-type manifolds with involution of type $(M, \mathcal{K})$ is locally projective.
	For the proof, we follow \cite[Lemma 2.7 and Theorem 5.6]{MR2047658}.
	
	\begin{lem}\label{l2-3-1-1}
		Let $f : \X \to S$ be a family of $K3^{[2]}$-type manifolds and let $\sigma \in H^0( S, R^2f_*\mathbb{Z} )$.
		If $\sigma(s) = \sigma|_{X_s} \in H^{1,1}( X_s, \mathbb{R} )$ for each $s \in S$,
		then we have $\sigma \in \operatorname{Im} \left( H^0( S, R^1f_*\mathcal{O}^*_{\X} ) \to H^0( S, R^2f_*\mathbb{Z} ) \right)$.
	\end{lem}
	
	\begin{proof}
		We may assume that $S$ is a polydisk.
		Since $h^2(X_s, \mathcal{O}_{X_s})=1$ for each $s \in S$, $R^2f_*\mathcal{O}_{\X}$ is an invertible sheaf and we may regard $R^2f_*\mathcal{O}_{\X}$ as a holomorphic line bundle on $S$.
		Since any holomorphic line bundle on the polydisk $S$ is trivial, there exists a nowhere-vanishing holomorphic section $\xi \in H^0(S, R^2f_*\mathcal{O}_{\X})$ such that
		$$
			R^2f_*\mathcal{O}_{\X} = \mathcal{O}_S \cdot \xi \quad \text{ and } \quad H^2(X_s, \mathcal{O}_{X_s} ) = \mathbb{C} \xi|_{X_s} \quad (s \in S).
		$$
		The exponential sequence on $\X$ 
		$$
			0 \to \mathbb{Z} \to \mathcal{O}_{\X} \to \mathcal{O}^*_{\X} \to 0 
		$$
		induces the following exact sequence:
		\begin{align}\label{al2-3-1-2}
			 H^0(S, R^1f_*\mathcal{O}^*_{\X} ) \xrightarrow{ \phi } H^0(S, R^2f_*\mathbb{Z}) \xrightarrow{ \psi } H^0(S, R^2f_*\mathcal{O}_{\X}).
		\end{align}
		Since $\psi( \sigma ) \in H^0(S, R^2f_*\mathcal{O}_{\X})$, there exists a holomorphic function $F$ on $S$ such that $\psi( \sigma ) = F \cdot \xi$.
		By our assumption of $\sigma|_{X_s} \in H^{1,1}( X_s, \mathbb{R} )$,
		$$
			0 = \psi( \sigma )|_{X_s} = F(s) \cdot \xi|_{X_s} \quad (s \in S).
		$$
		Therefore $F=0$ and $\sigma \in \operatorname{Ker} \psi$.
		By the exact sequence (\ref{al2-3-1-2}), we have $\sigma \in \operatorname{Im} \phi$.
	\end{proof}
	
	\begin{lem}\label{al2-3-1-3}
		Let $f : (\X, \iota) \to S$ be a family of $K3^{[2]}$-type manifolds with involution of type $(M, \mathcal{K})$.
		Then $f : \X \to S$ is locally projective.
	\end{lem}
	
	\begin{proof}
		Fix a point $s \in S$.
		By \cite[Proposition 4.4.(iv)]{MR3519981}, there exists a $\G$-equivariant ample line bundle $L_s$ on $X_s$.
		We may assume that $S$ is a polydisk.
		Choose an isomorphism $\alpha : R^2f_*\mathbb{Z} \to L_{2, S}$ such that for $s \in S$, $\alpha_s : H^2(X_s, \mathbb{Z}) \to L_2$ is an admissible marking for $(M, \mathcal{K})$.
		Then we have $\alpha_s (c_1(L_s)) \in M$.
		If $S$ is sufficiently small, it follows from Lemma \ref{l2-3-1-1} that there exists a holomorphic line bundle $\mathcal{L}$ on $\X$ such that $\mathcal{L}|_{X_s} =L_s$.
		Since ampleness is an open condition, there exists an open neighborhood $U \subset S$ of $s$ such that $\mathcal{L}|_{f^{-1}(U)}$ is relatively ample.  
	\end{proof}
	
	In particular, a family $f : (\X, \iota) \to S$ of $K3^{[2]}$-type manifolds with involution of type $(M, \mathcal{K})$ is locally \K.

	\begin{thm}\label{p-3-4}
		We define a real-valued function $\tau_{M, \mathcal{K}, \xs}$ on $S$ by 
		$$
			\tau_{M, \mathcal{K}, \xs}(s) =\tau_{M, \mathcal{K}}(X_s, \iota_s) \quad (s \in S).
		$$
		Then $\tau_{M, \mathcal{K}, \xs}$ is smooth and satisfies
		$$
			-dd^c \log \tau_{M, \mathcal{K}, \xs} = \frac{(t+1)(t+7)}{16} c_1(f_*K_{\xs}, h_{L^2}) +\omega_{H^{\cdot}(\xss)}.
		$$
	\end{thm}
	
	\begin{proof}
		By Proposition \ref{p-3-1}, we have
		\begin{align*}
			&[f_*(Td_{\iota}(T\xs, h_{\xs}) ch_{\iota}(\Omega^1_{\xs}, h^*_{\xs}))]^{(1,1)}  \\
			&= 2[f_*Td(T\xss, h_{\xss})]^{(1,1)} +dd^c \log A(\xs, h_{\xs})\\
				 &\quad  +\frac{t^2+7}{16} c_1(f_*K_{\xs}, h_{L^2}) +\frac{t^2+7}{16} dd^c \log Vol({\xs}, \omega_{\xs}) .
		\end{align*}
		Moreover, by Proposition \ref{p-3-2}, we have
		$$
			\sum_{q \geqq 0} (-1)^q [ch_{\iota} (R^qf_*\Omega^1_{\xs})]^{(1,1)} 
			= -\frac{t}{2} c_1(f_*K_{\xs}, h_{L^2})  +\frac{t}{2} dd^c \log Vol({\xs}, \omega_{\xs}) .
		$$
		Therefore, by the curvature formula for the equivariant Quillen metric (Theorem \ref{p-1-3}), we have
		\begin{align}\label{al3-3-5}
		\begin{aligned}
			dd^c\log \tau_{\iota}(\bar{\Omega}^1_{\xs}) =& 2\left[ f_* Td(\overline{T}_{\xss}) \right]^{(1,1)} +dd^c \log A(\xs, \omega_{\xs}) \\
			&+\frac{(t+7)(t+1)}{16} c_1(f_*K_{\xs}, h_{L^2}) +\frac{(t-7)(t-1)}{16}dd^c \log \vol \left( \xs, \omega_{\xs} \right).
		\end{aligned}
		\end{align}
		On the other hand, by Proposition \ref{p-3-3}, we have
		\begin{align*}
			\sum_{q \geqq 0} (-1)^q \left[ ch(R^q f_* \bar{\mathcal{O}}_{\XX}) \right]^{(1,1)}=& \frac{1}{2} dd^c \log \left\{ \vol(\xss, \omega_{\xss})^{-2} \vol_{L^2}(R^1 f_* \mathbb{Z}, \omega_{\xss})  \right\} \\
			& +\frac{1}{2} \omega_{H^{\cdot}(\xss)}.
		\end{align*}
		Therefore, by the curvature formula of the Quillen metric (Theorem \ref{p-1-3}), we have
		\begin{align}\label{al3-3-6}
		\begin{aligned}
			\left[ f_* Td(\overline{T}_{\xss}) \right]^{(1,1)} 
			=&\frac{1}{2}dd^c \log \left\{ \tau(\bar{\mathcal{O}}_{\mathscr{X}^{\iota}})^{-2} \vol(\xss, \omega_{\xss})^{-2} \vol_{L^2}(R^1 f_* \mathbb{Z}, \omega_{\xss}) \right\} \\
			&+\frac{1}{2} \omega_{H^{\cdot}(\xss)}.
		\end{aligned}
		\end{align}
		By (\ref{al3-3-5}) and (\ref{al3-3-6}) and by the definition of $\tau_{M, \mathcal{K}}$, we obtain the desired formula. 
	\end{proof}
	
	\begin{lem}\label{l4-3-1}
		The form $\omega_{H^{\cdot}(\xss)}$ is independent of the choice of an $\iota$-invariant fiberwise \K metric $\omega_{\xs}$.
	\end{lem}
	
	\begin{proof}
		Set
		\begin{align*}
			\lambda(\xss) = \bigotimes_{p,q} \Bigl( \det R^qf_*\Omega^p_{\xss} \Bigr)^{(-1)^{p+q}p},
		\end{align*}
		and 
		\begin{align*}
			\lambda_{dR}(\xss) = \bigotimes_{k=0}^{2n} \Bigl( \det R^kf_*\mathbb{C} \otimes \mathcal{O}_S \Bigr)^{(-1)^{k}k}.
		\end{align*}
		Then $\lambda(\xss)$ is a holomorphic line bundle on $S$, and $\lambda_{dR}(\xss)$ is a flat holomorphic line bundle on $S$.
		Moreover there exists a canonical isomorphism of smooth line bundles
		\begin{align}\label{al25}
			\lambda_{dR}(\xss) \cong \lambda(\xss) \otimes \overline{\lambda(\xss)}
		\end{align}
		induced by Hodge decomposition.
		
		For a sufficiently small open subset $U \subset S$, let $\tau \in H^0(U, \lambda(\xss))$ be a nowhere-vanishing holomorphic section of $\lambda(\xss)$ on $U$,
		and $\tau_0 \in \Gamma(U, \lambda_{dR}(\xss))$ be a non-zero flat section of $\lambda_{dR}(\xss)$ on $U$.
		By the isomorphism (\ref{al25}), there is a smooth function $f : U \to \mathbb{C}$ such that 
		$$
			\tau_0 = e^f \tau \otimes \bar{\tau}.
		$$
		 Note that the 2-form $dd^c \operatorname{Re}f$ is independent of the choice of the sections $\tau$ and $\tau_0$.
	 	The Hodge form $\omega_H$ is a real smooth differential form on $S$ of type $(1,1)$ defined by 
		$$
				\omega_H|_U = dd^c \operatorname{Re}f
		$$
		on $U$.
		By construction, this is independent of the choice of $\omega_{\xs}$.
	 
		By \cite[1.2.]{MR4406122}, the following formula holds:
		$$
			\omega_H= \frac{1}{2} \sum_{0 \leqq p,q \leqq 2} (-1)^{p+q} (p-q) c_1(R^qf_*\Omega^p_{\xss}, h_{L^2})=\omega_{H^{\cdot}(\xss)},
		$$
		which completes the proof.
	\end{proof}
	
	\begin{thm}\label{t-3-2}
		Let $(X, \iota)$ be a manifold of $K3^{[2]}$-type with antisymplectic involution of type $(M, \mathcal{K})$.
		Then $\tau_{M, \mathcal{K}}(X, \iota)$ is independent of the choice of an $\iota$-invariant \K metric.
		In particular, it is an invariant of $(X, \iota)$.
	\end{thm}
	
	\begin{proof}
		We follow \cite[Theorem 5.7]{MR2047658}.
		We regard $\mathbb{P}^1 = \mathbb{C} \cup \{ \infty \}$.
		Let $h_{X,0}$, $h_{X,\infty}$ be two $\iota$-invariant \K metrics on $X$.
		Set $\mathscr{X} =X \times \mathbb{P}^1$, and let $f:\mathscr{X} \to \mathbb{P}^1$ be the projection.
		For each $z \in \mathbb{P}^1$, we define a $\iota_z$-invariant \K metric $h_z$ on the fiber $X_z$ by
		$$
			h_z = \frac{1}{|z|^2+1} h_{X,0} +\frac{|z|^2}{|z|^2+1} h_{X,\infty}.
		$$
		We may regard $(h_z)_{z \in \mathbb{P}^1}$ as an $\iota$-invariant fiberwise \K metric on $T\X/\mathbb{P}^1$.\par
		Since the family $f:\mathscr{X} \to \mathbb{P}^1$ does not change the complex structures, we have
		$
			c_1(f_*K_{\X/\mathbb{P}^1}, h_{L^2}) = \omega_{H^{\cdot}(\XX/\mathbb{P}^1)} =0.
		$
		By Proposition \ref{p-3-4}, we have $dd^c \log \tau_{M, \mathcal{K}, \X/\mathbb{P}^1} =0$,
		and $\log \tau_{M, \mathcal{K}, \X/\mathbb{P}^1}$ is a pluriharmonic function on $\mathbb{P}^1$.
		Since $\mathbb{P}^1$ is compact, it is constant and $\tau_{M, \mathcal{K}, \X/\mathbb{P}^1}(0) = \tau_{M, \mathcal{K}, \X/\mathbb{P}^1}(\infty)$.
		Thus $\tau_{M, \mathcal{K}}(X, \iota)$ is independent of the choice of an $\iota$-invariant \K metric.
	\end{proof}
	
	By Lemma \ref{l-2-4}, the image of the period map $P_{M, \mathcal{K}} : \tilde{\mathcal{M}}_{M, \mathcal{K}} \to \mathcal{M}_{M, \mathcal{K}}$ is $\mathcal{M}_{M, \mathcal{K}}^{\circ}$.
	
	\begin{lem}\label{l-3-5}
		Let $p \in \mathcal{M}_{M, \mathcal{K}}^{\circ}$, and we define a real number $\tau_{M, \mathcal{K}}(p)$ by
		$$
			\tau_{M, \mathcal{K}}(p) =\tau_{M, \mathcal{K}}(X, \iota) \quad ((X, \iota) \in P_{M, \mathcal{K}}^{-1}(p)).
		$$
		Then $\tau_{M, \mathcal{K}}(p)$ is well-defined.
		Namely, it is independent of the choice of $(X, \iota) \in P_{M, \mathcal{K}}^{-1}(p)$.
	\end{lem}
	
	\begin{proof}
		Let $(X, \iota), (X', \iota') \in P_{M, \mathcal{K}}^{-1}(p)$, and let $\pi : (\mathscr{X}, \iota) \to \operatorname{Def}(X, \iota)$ and $\pi' : (\mathscr{X}', \iota') \to \operatorname{Def}(X', \iota')$ be the Kuranishi families, respectively.
		By Lemma \ref{l-2-5}, there exist sequences $\{ p_m \}^{\infty}_{m=1}$ in $\operatorname{Def}(X, \iota)$ and $\{ p'_m \}^{\infty}_{m=1}$ in $\operatorname{Def}(X', \iota')$ such that $\pi^{-1}(p_m) \cong \pi'^{-1}(p'_m)$ for all $m \geqq 1$ and $p_m \to p$, $p'_m \to p$ $(m \to \infty)$. 
		Therefore we have
		$$
			\tau_{M, \mathcal{K}}(X, \iota) =\lim_{m \to \infty} \tau_{M, \mathcal{K}, \mathscr{X}/\operatorname{Def}(X, \iota)}(p_m) = \lim_{m \to \infty} \tau_{M, \mathcal{K},\mathscr{X}'/\operatorname{Def}(X', \iota')}(p'_m) =\tau_{M, \mathcal{K}}(X', \iota'),
		$$
		which completes the proof.
	\end{proof}
	
	By Lemma \ref{l-3-5}, we obtain a real-valued smooth function $\tau_{M, \mathcal{K}}$ on $\mathcal{M}_{M, \mathcal{K}}^{\circ}$.
	
	Fix a vector $l \in M_{\mathbb{R}}$ with $l^2 \geqq 0$.
	Recall that $\Omega^+_{M^{\perp}}$ is a bounded symmetric domain of type IV.
	The Bergman metric $\omega_{\Omega^+_{M^{\perp}}}$ on $\Omega^+_{M^{\perp}}$ is defined by 
	\begin{align}\label{al5-3-2}
		\omega_{\Omega^+_{M^{\perp}}}([\eta]) =-dd^c \log B_{M^{\perp}}([\eta]) \quad ([\eta] \in \Omega^+_{M^{\perp}}),
	\end{align}
	where
	$$
		B_{M^{\perp}}([\eta]) =\frac{(\eta, \bar{\eta})}{|(\eta, l)^2|} \quad ([\eta] \in \Omega^+_{M^{\perp}}).
	$$
	Since $\omega_{\Omega^+_{M^{\perp}}}$ is $\Gamma_{M^{\perp}, \mathcal{K}}$-invariant, it induces an orbifold \K form $\omega_{\mathcal{M}_{M, \mathcal{K}}}$ on $\mathcal{M}_{M, \mathcal{K}}$.
	
	\begin{lem}\label{l6-3-1}
		There exists a smooth $(1,1)$-form $\sigma_{ M, \mathcal{K} }$ on $\mathcal{M}^{\circ}_{M, \mathcal{K}}$ such that for any $(X, \iota) \in \tilde{\mathcal{M}}_{M, \mathcal{K}}$ we have
		$$
			P_{M, \mathcal{K}}^*\sigma_{M, \mathcal{K}} = \omega_{H^{\cdot}(\XX/\operatorname{Def}(X, \iota))},
		$$
		where $P_{M, \mathcal{K}} : \operatorname{Def}(X, \iota) \to \mathcal{M}_{M, \mathcal{K}}$ is the period map of the Kuranishi family $\pi : (\X, \iota) \to \operatorname{Def}(X, \iota)$ of $(X, \iota)$.
	\end{lem}
	
	\begin{proof}
		Let $(X, \iota) \in \tilde{\mathcal{M}}_{M, \mathcal{K}}$ and let $\pi : (\mathscr{X}, \iota) \to \operatorname{Def}(X, \iota)$ be the Kuranishi family.
		Fix an isomorphism $\alpha : R^2\pi_* \mathbb{Z} \to L_{2, \operatorname{Def}(X, \iota)}$ such that for each $s \in \operatorname{Def}(X, \iota)$,
		$\alpha_s : H^2(X_s, \mathbb{Z}) \to L_2$ is an admissible marking for $(M, \mathcal{K})$.
		The period map $P_{M,\mathcal{K}} : \operatorname{Def}(X, \iota) \to \Omega^+_{M^{\perp}}$ is defined by
		$$
			P_{M,\mathcal{K}}(s) = \alpha_s( H^{2,0}(X_s) ) \quad (s \in \operatorname{Def}(X, \iota)).
		$$
		By the local Torelli theorem, we may assume that $P_{M,\mathcal{K}}$ is an isomorphism onto its image.
		We set $\operatorname{Def}(X, \iota, \alpha) = P_{M,\mathcal{K}} ( \operatorname{Def}(X, \iota) )$ 
		and we identify $\operatorname{Def}(X, \iota, \alpha)$ with $\operatorname{Def}(X, \iota)$.
		It suffices to show that $\{ \omega_{H^{\cdot}(\XX/\operatorname{Def}(X, \iota, \alpha))} \}_{(X, \iota, \alpha)}$ patch together to a smooth $(1,1)$-form $\tau$ on $\Omega^+_{M^{\perp}} \setminus \mathscr{D}_{M^{\perp}}$
		and that $\tau$ is $\Gamma_{M^{\perp}, \mathcal{K}}$-invariant.
		
		Let $(X', \iota', \alpha')$ be another triple and its Kuranishi family is denoted by $\pi' : (\mathscr{X}', \iota') \to \operatorname{Def}(X', \iota', \alpha')$.
		We assume that $\operatorname{Def}(X, \iota, \alpha) \cap \operatorname{Def}(X', \iota', \alpha') \neq \emptyset$.
		Set $U = \operatorname{Def}(X, \iota, \alpha) \cap \operatorname{Def}(X', \iota', \alpha')$ and $U^{\circ} = U \setminus \mathscr{D}_{\mathcal{K}}$.
		Here $\mathscr{D}_{\mathcal{K}}$ is a $\Gamma_{M^{\perp}, \mathcal{K}}$-invariant effective reduced divisor on $\Omega^+_{M^{\perp}}$ defined in Theorem \ref{l-2-1-A}.
		By Theorem \ref{l-2-1-A}, $\pi : (\pi^{-1}(U^{\circ}), \iota|_{\pi^{-1}(U^{\circ})}) \to U^{\circ}$ and $\pi' : ((\pi')^{-1}(U^{\circ}), \iota'|_{(\pi')^{-1}(U^{\circ})}) \to U^{\circ}$ are isomorphic.
		By Lemma \ref{l4-3-1},
		$$
			\omega_{H^{\cdot}(\XX/\operatorname{Def}(X, \iota, \alpha))}|_{U^{\circ}} =\omega_{H^{\cdot}((\X')^{\iota} / \operatorname{Def}(X', \iota', \alpha'))}|_{U^{\circ}}.
		$$
		Since $U^{\circ}$ is dense in $U$ and since both $\omega_{H^{\cdot}(\XX/\operatorname{Def}(X, \iota, \alpha))}$ and $\omega_{H^{\cdot}((\X')^{\iota} / \operatorname{Def}(X', \iota', \alpha'))}$ are smooth, we have
		$$
			\omega_{H^{\cdot}(\XX/\operatorname{Def}(X, \iota, \alpha))}|_{U} =\omega_{H^{\cdot}((\X')^{\iota} / \operatorname{Def}(X', \iota', \alpha'))}|_{U}.
		$$
		Since $\{ \operatorname{Def}(X, \iota, \alpha) \}$ is an open covering of $\Omega^+_{M^{\perp}} \setminus \mathscr{D}_{M^{\perp}}$, there exists a smooth $(1,1)$-form $\tau$ on $\Omega^+_{M^{\perp}} \setminus \mathscr{D}_{M^{\perp}}$ such that
		$$
			\tau|_{\operatorname{Def}(X, \iota, \alpha)} = \omega_{H^{\cdot}(\XX/\operatorname{Def}(X, \iota, \alpha))}
		$$
		for each $(X, \iota, \alpha)$.
		
		For $\gamma \in \Gamma_{M^{\perp}, \mathcal{K}}$, there exists $g \in \Gamma(\mathcal{K})$ such that $\gamma = g|_{M^{\perp}}$.
		Since 
		$$
			\gamma^* \omega_{H^{\cdot}(\XX/\operatorname{Def}(X, \iota, \alpha))} = \omega_{H^{\cdot}(\XX/\operatorname{Def}(X, \iota, g^{-1} \circ \alpha))},
		$$ 
		the form $\tau$ is $\Gamma_{M^{\perp}, \mathcal{K}}$-invariant.
		Therefore $\tau$ induces a smooth $(1,1)$-form $\sigma_{ M, \mathcal{K} }$ on $\mathcal{M}^{\circ}_{M, \mathcal{K}}$.
	\end{proof}
	
	\begin{thm}\label{t5-3-1}
		The following equation of differential forms on $\mathcal{M}^{\circ}_{M, \mathcal{K}}$ holds:
		$$
			-dd^c \log \tau_{M, \mathcal{K}, \mathscr{X}/S} = \frac{(t+1)(t+7)}{8} \omega_{\mathcal{M}_{M, \mathcal{K}}} +\sigma_{ M, \mathcal{K} }.
		$$
	\end{thm}
	
	\begin{proof}
		Let $(X, \iota, \alpha)$ be as in Lemma \ref{l6-3-1}.
		Applying Theorem \ref{p-3-4} to the Kuranishi family $\pi : (\mathscr{X}, \iota) \to \operatorname{Def}(X, \iota, \alpha)$,
		the following equation on $\operatorname{Def}(X, \iota, \alpha)$ holds:
		$$
			-dd^c \log \tau_{M, \mathcal{K}, \X/\operatorname{Def}(X, \iota, \alpha)} = \frac{(t+1)(t+7)}{16} c_1(\pi_*K_{\X/\operatorname{Def}(X, \iota, \alpha)}, h_{L^2}) +\omega_{H^{\cdot}(\XX/\operatorname{Def}(X, \iota, \alpha))}.
		$$
		By (\ref{al5-3-2}), we have $c_1(\pi_*K_{\X/\operatorname{Def}(X, \iota, \alpha)}, h_{L^2})  =2 \omega_{\Omega^+_{M^{\perp}}}|_{\operatorname{Def}(X, \iota, \alpha)}$.
		By Lemma \ref{l6-3-1}, we obtain the desired equation.
	\end{proof}
	
	\begin{rem}\label{r5-3-3}
		Let $(X, \iota)$ be a manifold of $K3^{[2]}$-type with involution of type $(M, \mathcal{K})$ and let $f : (\X, \iota) \to S$ be a deformation of $(X, \iota)$.
		Consider the equivariant analytic torsion $\tau_{\iota}(\overline{\mathcal{O}_{\X}})$ of trivial line bundle on $\X$ with the canonical metric.
		In the same manner as in this subsection \S 3.3., we have
		$$
			 dd^c \log \left\{ \tau_{\iota} (\overline{\mathcal{O}_{\X}}) A_0(\xs, h_{\xs}) \right\} =0,
		$$
		where $A_0(\xs, h_{\xs})$ is a function on $S$ defined by
		\begin{align*}
			A_0 (\xs, h_{\xs})&(s) = A_0(X_s, \iota_s, h_s) :=\exp 
				\left[
					\frac{1}{96} \int_{X_s^{\iota}}  \log  \left\{  
							\frac{\omega_s^4/4!}{\eta_s^2 \wedge \bar{\eta_s}^2} \frac{||\eta_s^2||^2_{L^2}}{Vol(X_s, \omega_s)}  
						\right\} \Omega_0|_{X_s^{\iota}} \right] ,\\
			\biggl. \Omega_0 = &\ax^2 -2\ay -\au^2 +3\av  \biggr. .
		\end{align*}
		Therefore $\tau_{\iota}(\overline{\mathcal{O}_{\X}})$ does not induce an interesting invariant which reflects the complex structure of a $K3^{[2]}$-type manifold with antisymplectic involution.
		
		Similarly, consider the equivariant analytic torsion $\tau_{\iota}(\overline{\Omega^2_{\xs}})$ of $\Omega^2_{\xs}$ on $\X$ with the $\G$-invariant hermitian metric induced from the fiberwise \K metric $h_{\xs}$.
		In the same manner as in this subsection \S 3.3., we have
		$$
			 dd^c \log \tau_{\iota}(\overline{\Omega^2_{\xs}}) =0.
		$$
		Therefore, $\tau_{\iota}(\overline{\Omega^2_{\xs}})$ does not induce an interesting invariant which reflects the complex structure of a $K3^{[2]}$-type manifold with antisymplectic involution.
		
		It is very likely that $\tau_{\iota}( \overline{\mathcal{O}_{X}} ) A_0(X, \iota, h_X)$ and $\tau_{\iota} ( \overline{\Omega^2_X} )$ are constant functions on the space $\mathcal{M}_{M, \mathcal{K}}^{\circ}$.
		These are the reason why we consider the equivariant analytic torsion $\tau_{\iota}(\overline{\Omega^1_{X}})$ of holomorphic cotangent bundle with $\G$-invariant hermitian metric induced from the fiberwise \K metric $h_{X}$.
	\end{rem}
	
	As an application of this invariant, we show the isotriviality of families of $K3^{[2]}$-type manifolds with antisymplectic involution.
	We assume the following properties of $(M, \mathcal{K})$:
	\begin{itemize}
		\item $t \neq -1, -7$.
		\item Each manifold of $K3^{[2]}$-type $(X, \iota)$ with antisymplectic involution of type $(M, \mathcal{K})$ satisfies $q(X^{\iota})=p_g(X^{\iota})=0$.
	\end{itemize} 
	Here $q(X^{\iota})= \dim H^1(X^{\iota}, \mathcal{O}_{X^{\iota}})$ be the irregularity of $X^{\iota}$,
	and $p_g(X^{\iota}) = \dim H^2(X^{\iota}, \mathcal{O}_{X^{\iota}})$ be the geometric genus of $X^{\iota}$.
	
	\begin{exa}\label{e3-3-1}
		Let $M_0= U(2) \oplus E_8(2) \text{ or } \Lambda_k(2)^{\perp}$ $(k=0, \dots, 9)$, 
		where 
		$$
			\Lambda_k = I_2 \oplus -I_{10-k} \quad (k \neq 8), \qquad \Lambda_8 = I_2 \oplus -I_{2} \text{ or } U \oplus U \quad (k = 8).
		$$
		By \cite[Theorem 4.2.2.]{MR633160}, the fixed locus $Y^{\sigma}$ of a 2-elementary K3 surface $(Y, \sigma)$ of type $M_0$ is the empty set or consists of smooth rational curves.
		Let $M=M_0 \oplus \mathbb{Z}e$ and let $\mathcal{K} \in \ktm$ be a natural chamber.
		Then $(M, \mathcal{K})$ satisfies the above assumption.
	\end{exa}
	
	\begin{thm}\label{t3-3-1}
		Suppose that $(M, \mathcal{K})$ satisfies the above assumption.
		Then there exists no irreducible projective curve on $\mathcal{M}_{M, \mathcal{K}}^{\circ}$.
		Moreover if $f : (\mathscr{X}, \iota) \to S$ is a family of $K3^{[2]}$-type manifolds with antisymplectic involution of type $(M, \mathcal{K})$ and $S$ is compact,
		then $f$ is isotrivial.
		Namely, any two fibers of $f$ are isomorphic.
	\end{thm}
	
	\begin{proof}
		Let $(X, \iota) \in \tilde{\mathcal{M}}_{M, \mathcal{K}}$ and let $\pi : (\mathscr{X}, \iota) \to \operatorname{Def}(X, \iota)$ be the Kuranishi family.
		Since $q(X_s^{\iota})=p_g(X_s^{\iota})=0$ for each fiber $(X_s, \iota_s)$, we have $\omega_{H^{\cdot}(\XX/\operatorname{Def}(X, \iota))} =0$.
		By the construction of $\sigma_{ M, \mathcal{K} }$, we have $\sigma_{ M, \mathcal{K} } =0$ on $\mathcal{M}^{\circ}_{M, \mathcal{K}}$.
		By Theorem \ref{t5-3-1}, we have the formula
		\begin{align}\label{al3-3-7}
			-dd^c \log \tau_{M, \mathcal{K}} = \frac{(t+1)(t+7)}{8} \omega_{\mathcal{M}_{M, \mathcal{K}}}.
		\end{align}
		on $\mathcal{M}_{M, \mathcal{K}}^{\circ}$.
		Since $\omega_{\mathcal{M}_{M, \mathcal{K}}}$ is a positive form, $-\frac{8}{(t+1)(t+7)} \log \tau_{M, \mathcal{K}}$ is a plurisubharmonic function on $\mathcal{M}_{M, \mathcal{K}}^{\circ}$.\par
		Suppose that there is an irreducible projective curve $C$ on $\mathcal{M}_{M, \mathcal{K}}^{\circ}$.
		Then the function $\left(-\frac{8}{(t+1)(t+7)} \log \tau_{M, \mathcal{K}} \right)|_C$ is a subharmonic function on $C$.
		By the maximal principle, it is a constant function, which contradicts the formula (\ref{al3-3-7}) since $\omega_{\mathcal{M}_{M, \mathcal{K}}}$ is K\"ahler.\par 
		Let $f : (\mathscr{X}, \iota) \to S$ be a family of $K3^{[2]}$-type manifolds with antisymplectic involution of type $(M, \mathcal{K})$.
		Assume that $S$ is compact.
		The period map $P_{M, \mathcal{K}} : S \to \mathcal{M}_{M, \mathcal{K}}^{\circ}$ is a proper holomorphic map.
		By Remmert proper mapping theorem, its image $P_{M, \mathcal{K}}(S)$ is an analytic subset of the quasi-projective variety $\mathcal{M}_{M, \mathcal{K}}^{\circ}$.
		Suppose that $f$ is not isotrivial.
		Then $P_{M, \mathcal{K}}$ is not a constant map and $P_{M, \mathcal{K}}(S)$ contains an irreducible projective curve.
		This contradicts to the first statements.
	\end{proof}

\bibliography{Reference}
\bibliographystyle{plain}

\Address

\end{document}